\documentclass[11pt]{amsart}

\usepackage{amsmath, amssymb, amsfonts,enumerate}
\usepackage[all]{xy}
\usepackage{amscd,stmaryrd}
\usepackage{comment}
\usepackage{mathtools}
\usepackage{tikz} 
\usepackage{tikz-cd} 
\tikzcdset{diagrams={nodes={inner sep=7.5pt}}}

\usepackage{url}
\usepackage{hyperref}

\newtheorem{theorem}{Theorem}[section]
\newtheorem{lemma}[theorem]{Lemma}
\newtheorem{corollary}[theorem]{Corollary}

\newtheorem{theoremletter}{Theorem}

 \theoremstyle{definition}
 \newtheorem{definition}[theorem]{Definition} 
 \newtheorem{remark}[theorem]{Remark}

 \newtheorem{example}[theorem]{Example}

  \newtheorem*{example*}{Example}

\numberwithin{equation}{section}
 
\newcommand {\N}{\mathbb{N}} 
\newcommand {\Z}{\mathbb{Z}}

\DeclareMathOperator{\Ker}{Ker}

\DeclareMathOperator{\im}{Im}

\DeclareMathOperator{\Id}{Id}

\begin{document}
\title[Asynchronous non-uniform cellular automata]{On reversible asynchronous non-uniform cellular automata}   
\author[Xuan Kien Phung]{Xuan Kien Phung}
\address{Département d'informatique et de recherche opérationnelle,  Université de Montréal, Montréal, Québec, H3T 1J4, Canada.}
\email{phungxuankien1@gmail.com}   

\begin{abstract}
We study the class of asynchronous non-uniform cellular automata (ANUCA) over an arbitrary group universe with multiple local transition rules.
We introduce the notion of stable injectivity, stable reversibility,  stable post-surjectivity and investigate several dynamical properties of such automata. In particular, we establish the equivalence between reversibility, stable reversibility, and stable injectivity for ANUCA. We also prove  the invertibility of several classes of injective and stably injective ANUCA. Counter-examples are given to highlight the differences between cellular automata and ANUCA. 
\end{abstract}
\date{\today}
\maketitle
  
\setcounter{tocdepth}{1}

\section{Introduction}  

We briefly recall notions of symbolic dynamics. 
Given a set $A$ equipped with the discrete topology and a group $G$, 
a \emph{configuration} $c \in A^G$ is simply a map $x \colon G \to A$.  Two configurations $x,y  \in A^G$ are \emph{asymptotic} if for some finite subset $E \subset G$, $x\vert_{G \setminus E}=y\vert_{G \setminus E}$.  
The \emph{Bernoulli shift} action $G \times A^G \to A^G$ is defined by $(g,x) \mapsto g x$, 
where $(gx)(h) \coloneqq  x(g^{-1}h)$ for  $g,h \in G$ and $x \in A^G$. 
The space $A^G$ is equipped with the \emph{prodiscrete topology}, i.e., the smallest
topology for which the projection $A^G \to A^{\{g\}}$  is continuous for every $g \in G$. When $A$ is finite, the Tychonoff theorem implies that $A^G$ is a compact Hausdorff space. For every $x \in A^G$, we denote  
\begin{equation*}
    \Sigma(x) \coloneqq \overline{\{gx \colon g \in G\}} \subset A^G. 
\end{equation*}
\par
Following von Neumann \cite{neumann}, a \emph{cellular automaton} (CA)  over a group $G$, the \emph{universe}, and a set $A$, the \emph{alphabet}, is a map
$\tau \colon A^G \to A^G$ admitting a finite \emph{memory set} $M \subset G$
and a \emph{local defining map} $\mu \colon A^M \to A$ such that 
\begin{equation*} 
\label{e:local-property}
(\tau(x))(g) = \mu((g^{-1} c )\vert_M)  \quad  \text{for all } x \in A^G \text{ and } g \in G.
\end{equation*} 
\par 
Equivalently, a map $\tau \colon A^G \to A^G$ is a CA if and only it is $G$-equivariant and uniformly continuous (cf.~\cite{hedlund-csc}, \cite{hedlund}).  Various physical phenomena in fluid dynamics as well as biological organisms consisting of homogeneous cells whose evolution is described by the same local transition rule can be simulated by CA. In computer science, CA are studied as powerful models
of computation. For instance,  the famous two-dimensional CA Game of Life defined by Conway \cite{GOL} is Turing complete. The mathematical theory of CA also leads to a surprising characterization of amenable groups (see Section~\ref{s:sofic-groups}) in terms of the Garden of Eden theorem for CA (cf. \cite{moore}, \cite{myhill}, \cite{ceccherini}, \cite{bartholdi-kielak}).  Amenable groups were also  invented by von Neumann \cite{neumann-amenable}.  
\par  
To generalize the notion of CA, we introduce the class of asynchronous non-uniform cellular automata (ANUCA) which allows multiple local defining maps. Our definition extends the class of $\nu$-CA studied in \cite{Den-12a}, \cite{Den-12b} where the group  universe is $\Z$ or $\Z^d$ for $d \geq 2$.  

\begin{definition}
\label{d:most-general-def-asyn-ca}
Let $G$ be a group and let  $A$ be a set. Let $M \subset G$ be a subset and let $S = A^{A^M}$ be the set of all maps $A^M \to A$. Given $s \in S^G$, the  \emph{asynchronous non-uniform cellular automaton} (ANUCA) $\sigma_s \colon A^G \to A^G$ associated with $s$ is defined for all $x \in A^G$, $g \in G$ by the formula: 
\begin{equation*} 
\sigma_s(x)(g)\coloneqq  
    s(g)((g^{-1}x)  
	\vert_M).
\end{equation*} 
\par 
The set $M$ is called the \emph{memory} of $\sigma_s$.  
The ANUCA $\sigma_s$ is said to be \emph{stably injective} if  $\sigma_p$ is injective  for every $p \in \Sigma(s)$. We say that $\sigma_s$ is \emph{invertible} if it is bijective and the inverse map  $\sigma_s^{-1}$ is an ANUCA with \emph{finite} memory. 
\par 
We say that $\sigma_s$ is \textit{reversible}, or \emph{left-invertible}, if  there exists an ANUCA with \textit{finite} memory $\tau \colon A^G \to A^G$ such that $\tau \circ \sigma= \Id$. Moreover, $\sigma_s$ is \emph{stably reversible} if there exist $N\subset G$ finite and $q \in T^G$ where $T=A^{A^N}$ such that for every $p \in \Sigma(s)$, we can find $w \in \Sigma(q)$ such that $\sigma_w \circ \sigma_p=\Id$.
\end{definition} 
\par
The configuration $s \in S^G$ is called the \textit{configuration of local defining maps} of the ANUCA  $\sigma_s$. Note that every CA is an ANUCA with finite memory and constant configuration of local defining maps. Moreover, a CA is injective if and only if it is stably injective since   $\Sigma(c)=\{c\}$ for all constant configuration~$c$. 
It is clear from the above definition for ANUCA that 
\begin{align}
    \label{e:intro-equiv-1}
\text{Stable reversibility}\Rightarrow\text{Reversibility}\Rightarrow\text{Injectivity}\Leftarrow\text{Stable injectivity}.  
\end{align}
\par 
The first main result of the paper is the following characterization of reversible ANUCA in terms of stable injectivity (see Section~\ref{s:direct-finintess}).  

 \begin{theoremletter}
\label{t:intro-characterization-reversible-ANUCA} 
Let $M$ be a finite subset of a countable group~$G$. Let $A$ be a finite set and let $S=A^{A^M}$. Let $s \in S^G$. Then the following are equivalent: 
\begin{enumerate} [\rm (i)]
    \item $\sigma_s$ is reversible; 
    \item $\sigma_s$ is stably reversible;  
    \item $\sigma_s$ is stably injective. 
\end{enumerate}
\end{theoremletter}
\par 
Hence, the first two and the last properties in \eqref{e:intro-equiv-1} are in fact equivalent and Example~\ref{ex:1}-\ref{ex:2} show that are strictly stronger than Injectivity.   
\par 
Our next results concern    Gottschalk's surjunctivity conjecture \cite{gottschalk} which asserts that over any universe, any injective CA with finite alphabet must be surjective. 
Over sofic group universes, the surjunctivity conjecture was famously shown by Gromov-Weiss \cite{gromov-esav},  \cite{weiss-sgds} (see also \cite{csc-sofic-linear},  \cite{elek}, \cite{phung-geometric},  \cite{phung-weakly}). The class of sofic groups was introduced by Gromov  \cite{gromov-esav} as a  generalization of residually finite groups and amenable groups (see Section~\ref{s:sofic-groups}).
\par  
The situation for ANUCA is more complicated. Example~\ref{ex:3} gives a simple one-dimensional stably injective ANUCA which is not surjective.  However, when the configuration of local defining maps of an ANUCA is asymptotic to a constant configuration, we establish the   \emph{invertibility}, i.e. bijectivity plus reversibility, for stably injective ANUCA and even for injective ANUCA over a suitable universe (see Section~\ref{s:sofic-groups}, Theorem~\ref{c:application-disturbed-ca}, Theorem~\ref{c:application-disturbed-ca-amenable}).  
\par 
\begin{theoremletter}
\label{t:main-B}
Let $M$  be a finite subset of a countable group $G$. Let $A$ be a finite set and let $S=A^{A^M}$. Let $s \in S^G$ be asymptotic to a constant configuration $c \in S^G$. Then $\sigma_s$ is invertible in each of the following cases:
\begin{enumerate} [\rm (i)]
\item $G$ is amenable, e.g. an abelian group, and $\sigma_s$ is injective;  
\item $G$ is residually finite, e.g. a free group, and $\sigma_s, \sigma_c$ are injective.
\end{enumerate}
\end{theoremletter}
\par 
 
To illustrate the case of higher dimensional ANUCA, i.e., when the universe is $\Z^d$, we can prove the invertibility of  stably injective ANUCA where the configuration of local defining maps is no longer required to be asymptotically constant but we only need the mild condition of \emph{bounded singularity} (see  Section~\ref{s:bounded-singularity-section}).

\begin{theoremletter}
\label{t:intro-singularity}
Let $M \subset \Z^d$ be a finite subset. Let $A$ be a finite set and let $S=A^{A^M}$.  Suppose that  for some $s \in S^{\Z^d}$, the ANUCA  $\sigma_s$ is stably injective with  bounded singularity. 
Then $\sigma_s$ is invertible. 
\end{theoremletter}
\par 
For instance, fix $d, n_0 \geq 1$ and two functions $f,g \colon \N \to \N$ such that 
\[
 \lim_{n \to \infty} f(n)= \lim_{n \to \infty} g(n) =\lim_{n \to \infty} f(n) - g(n)= \infty. 
 \]
\par 
Let $R_n = \{ a \in \Z\colon g(n) \leq \vert a \vert  \leq f(n)\}^d \subset \Z^d$ for $n \geq n_0$. Let $A$ be a finite set and let $M\subset \Z^d$ be finite. Then for every $s \in S^{\Z^d}$, where  $S= A^{A^{M}}$, such that $s$ is constant on $R_n$ for every $n \geq n_0$, the ANUCA $\sigma_s$ has bounded singularity. By Theorem~\ref{t:intro-singularity}, such an ANUCA is automatically invertible whenever it is stably injective. 
\par 
The paper is organized as follows. 
We introduce various induced local maps of ANUCA in Section~\ref{s:induced-local-map}. Then we investigate the continuity and closedness properties of ANUCA in  Section~\ref{s:continuity-closedness}. In Section~\ref{s:stably-injective}, we formulate  an almost equivariance property   and several invariant properties of ANUCA under the translations of the configurations of local defining maps (Lemma~\ref{l:pseudo-equivarian}). In Section~\ref{s:monoid}, we show that ANUCA over the same universe and alphabet form a monoid under the composition of maps (Theorem~\ref{t:monoid-asynchronous}). 
\par 
We shall establish various stable properties of stably injective ANUCA through a series of results:  Theorem~\ref{t:reversible-finite-alpha-beta},  Theorem~\ref{c:application-disturbed-ca-amenable}, and  Theorem~\ref{t:direct-post-injective}. 
The implication (iii)$\implies$(ii) of Theorem~\ref{t:intro-characterization-reversible-ANUCA} is proved in Theorem~\ref{t:reversible-finite-alpha-beta}.  
We shall prove Theorem~\ref{t:main-B} as a direction consequence of Theorem~\ref{c:application-disturbed-ca} and  Theorem~\ref{c:application-disturbed-ca-amenable} given in 
 Section~\ref{s:app-disturbed-sofic} and   Section~\ref{s:app-disturbed-amenable}. The proof of Theorem~\ref{t:intro-singularity} is given in Section~\ref{s:bounded-singularity-section}. In Section~\ref{s:direct-finintess}, we show that every reversible ANUCA is stably injective (Theorem~\ref{t:direct-post-injective}) which completes the proof of Theorem~\ref{t:intro-characterization-reversible-ANUCA}. 
 \par 
In Section~\ref{s:pointwise-uniform-post}, we show that over a countable universe, 
every pre-injective post-surjective ANUCA with finite memory is invertible (Theorem~\ref{t:invertible-general}). Moreover, we prove in Lemma~\ref{l:uniform-psot-surjectivity} the pointwise uniform post-surjectivity property of post-surjective ANUCA.   
We explore in  Section~\ref{s:stably-post0surjective} the notion of stably post-surjectivity and establish a  uniform post-surjectivity property of the family $\{\sigma_p \colon p \in \Sigma(s)\}$ associated with a stably post-surjective $\sigma_s$. 
\par 
Finally, we present several counter-examples  in Section~\ref{s:counter} to show that our results are in a sense  optimal.  Examples~\ref{ex:1} describes an ANUCA which is injective but not stably injective, not surjective, and not reversible.   Example~\ref{ex:2} gives a bijective non reversible (and thus non stably injective) ANUCA. In particular, for ANUCA, there is no implication between bijectivity and reversibility.

\section{Amenable and residually finite groups}
\label{s:sofic-groups}

\subsection{Amenable groups} 
Although we shall not need the precise definition of amenable groups, we include a simple  characterization for the sake of completeness. A group $G$ is amenable if it satisfies the Følner's condition \cite{folner}: 
for every $\varepsilon >0$ and $T \subset G$ finite, there exists $F \subset G$ finite such that $\vert TF \vert \leq (1+\varepsilon)\vert F \vert$. Finitely generated groups of subexponential growth and  solvable groups are amenable,  while all groups containing a subgroup isomorphic to a free group of rank 2 are not amenable. See e.g. \cite{stan-amenable} for some more details.
\par 
The celebrated Garden of Eden of Moore and Myhill \cite{moore}, \cite{myhill} states that a CA with finite alphabet over the universe $\Z$ is pre-injective if and only if it is surjective. In \cite{ceccherini},  the Garden of Eden theorem was generalized to hold over amenable group universes (see also \cite{cscp-alg-goe}, \cite{phung-2020},  \cite{phung-post-surjective}). 

\subsection{Residually finite groups} 

We recall that a group  is \emph{residually finite} if the intersection of its  finite-index subgroups
is reduced to the identity element. 
Equivalently, a group $G$ is residually finite if for every finite subset $E \subset G$, there exists a finite group and a surjective group homomorphism $\varphi \colon G \to H$ such that  $\varphi\vert_E \colon E \to H$ is injective.
\par 
Notable examples of residually finite groups include  finitely generated abelian group, e.g., $\Z^d$ and cyclic groups, and 
more generally finitely generated linear groups due to a theorem of Mal’cev. 
\par 
We formulate below a technical property of the  class  of residually finite groups which will be useful in Section~\ref{s:app-disturbed-sofic}. 

\begin{lemma}
\label{l:separation-forte-res}
Let $G$ be a residually finite group.  Then for all finite subsets $M,E \subset G$, there exist a finite subset $K \subset G$, a finite group $H$, and a surjective group homomorphism $\varphi \colon G \to H$ such that: 
\begin{enumerate}[\rm (a)]
\item $E \cup M \subset K$; 
\item $\varphi\vert_K \colon K \to H$ is bijective; 
\item 
$\varphi(KM \setminus K) \cap \varphi(E) = \varnothing$. 
\end{enumerate}
\end{lemma}
\par 
\begin{proof}
Up to replacing $E$ and $M$ by the finite set $\{1_G\} \cup E^{-1}\cup E \cup M^{-1}\cup M$, we can suppose without loss of generality that $E=M$, $E=E^{-1}$, and moreover $1_G \in E$. Since $G$ is residually finite, we can find a finite group $H$ and a surjective homomorphism $\varphi \colon G \to H$ such that the restriction $\varphi\vert_{E^2} \colon E^2 \to H$ is injective. 
\par
Consider the finite index subgroup $Z= \Ker \varphi \subset G$ and let $K$ be a complete set of representatives of the right cosets of $Z$ in $G$ such that $E^2 \subset K$. Hence, the condition (b) is satisfied. Note that the choice of $K$ also satisfies (a) since $E \subset E^2 \subset K$ as $1_G \in E$. 
\par 
Let us check the condition (c). Suppose that $k \in K$,  $z \in Z$, and $x,y \in E$ verify $kx = zy$. Then $k=zyx^{-1}$ and it follows that $k$ and $yx^{-1}$ belong to the same right coset of $Z$. On the other hand, both $k$ and $yx^{-1}$ are  elements of $K$ since $yx^{-1} \in E^2 \subset K$. We deduce from the choice of $K$ that $k= yx^{-1}$. Consequently, $z=1_G$ and thus $kx =y \in E \subset K$. 
\par 
Therefore, we have proved that $KE \cap Z E \subset K$, which is exactly equivalent to the condition (c). The proof is thus complete. 
\end{proof}

\section{Induced local maps of ANUCA} 
\label{s:induced-local-map}
\subsection{Local maps for arbitrary group universes}
\label{s:induced-local-map-1}
Let $G$ be a group and let $A$ be a set. For every subset $E\subset G$ and $x \in A^E$ we define  $gx \in A^{gE}$ by setting $gx(gh)=x(h)$ for all $h \in E$. In particular, we find that   $gA^E=A^{gE}$. 
\par 
Let $M$  be a subset of a group $G$. Let $A$ be a set and let $S=A^{A^M}$ be the collection of all maps $A^M \to A$. 
\par 
For every finite subset $E \subset G$  and $w \in S^{E_n}$,  
we define a map  $f_{E,w}^+ \colon A^{E M} \to A^{E}$ defined as follows. For every $x \in A^{EM}$ and $g \in E$, we set: 
\begin{align}
\label{e:induced-local-maps} 
    f_{E,w}^+(x)(g) & = w(g)((g^{-1}x)\vert_M). 
\end{align}
\par 
In the above formula, note that  $g^{-1}x \in A^{g^{-1}EM}$ and $M \subset g^{-1}EM$ since $1_G \in g^{-1}E$ for $g \in E$. Therefore, the map   $f_{E,w}^+ \colon A^{E M} \to A^{E}$ is well defined. 
\par
Consequently, for every $s \in S^G$, we have a well-defined induced local map $f_{E, s\vert_E}^+ \colon A^{E M} \to A^{E}$ for every finite subset $E \subset G$ which satisfies: 
\begin{equation}
\label{e:induced-local-maps-general} 
    \sigma_s(x)(g) =  f_{E, s\vert_E}^+(x\vert_{EM})(g)
\end{equation}
for every $x \in A^G$ and $g \in E$. Equivalently, we have for all $x \in A^G$ that: 
\begin{equation}
\label{e:induced-local-maps-proof} 
    \sigma_s(x)\vert_E =  f_{E, s\vert_E}^+(x\vert_{EM}). 
\end{equation}

\subsection{Local maps and reversible ANUCA} 
\label{s:induced-local-map-2}
For the notation, let $M, K$ be finite subsets of a  group $G$. Let $A$ be a finite set and let $s,t \in S^G$ where $S=A^{A^M}$. Let  $\varphi \colon G \to H$ be a surjective group homomorphism where $H$ is a finite group such that $\varphi \vert_K \colon K \to H$ is a bijection. In other words, $K$ forms a complete set of representatives of the right cosets of the subgroup $\Ker \varphi$ in $G$.  
\par 
The configurations $s,t$ induce
the maps $\Psi_{K,s}$ and $\Psi_{K,t} \colon A^K \to A^K$ defined as follows. Given $x \in A^K$, we determine $\tilde{x} \in A^G$ by setting  $\tilde{x}(g)=x(k_g)$ for all $g \in G$ and the unique $k_g \in K$ such that $\varphi(g) = \varphi (k_g)$. Equivalently, $\tilde{x}(hk) = x(k)$ for all $k \in K$ and $h \in \Ker \varphi$. 
Then we put 
\begin{align}
    \label{e:local-periodic-proof-1}
\Psi_{K,s}(x) \coloneqq \sigma_s(\tilde{x})\vert_K, \quad  
\Psi_{K,t}(x) \coloneqq \sigma_t(\tilde{x})\vert_K.
\end{align}

\begin{lemma}
\label{l:induced-local-map-reversible-1}
With the above notation and hypotheses, suppose in addition that $\sigma_t \circ \sigma_s = \Id$ and $s(g)=s(k_g)$ for all $g \in KM\setminus K$. Then one has:  
\[
\Psi_{K,t} \circ \Psi_{K,s}=\Id.
\] 
\end{lemma}

\begin{proof}
Let $x \in A^K$. Consider   $u=f^+_{KM, s\vert_{KM}}  (\tilde{x}\vert_{KM^2}) \in A^{KM}$ and  
$v=u\vert_K$. Then we deduce from \eqref{e:local-periodic-proof-1} and \eqref{e:induced-local-maps-proof} that
\begin{equation}
    \label{e:induced-local-maps-proof-1-2-3}
v= u\vert_K = \sigma_s(\tilde{x})\vert_K = \Psi_{K,s}(x). 
\end{equation}
\par 
We claim that  $u=\tilde{v}\vert_{KM}$. Indeed, let us  fix   $g \in KM$. 
If $g \in K$ then there is nothing to prove since $\tilde{v}\vert_K=v=u\vert_K$.  
Suppose now that $g \in KM \setminus K$. 
\par 
On the one hand, for the unique $k_g \in K$ and $h_g \in \Ker \varphi$ such that $g= h_g k_g$, we have: 
\begin{align} 
\label{e:lacol-map}
\tilde{v}(g) = v(k_g) = u(k_g) = s(k_g)((k_g^{-1}\tilde{x})\vert_{M}). 
\end{align}
\par 
On the other hand, observe that $(g^{-1}\tilde{x})\vert_{M}=(k_g^{-1}\tilde{x})\vert_{M} $ since for every $m \in M$, we have by the definition of $\tilde{x}$ that: 
\[
\tilde{x}(gm)=\tilde{x}(h_gk_gm)=\tilde{x}(k_gm).
\]
\par 
Consequently, as $s(g)=s(k_g)$ by hypothesis, we deduce that: 
\begin{align} 
\label{e:lacol-map-2}
u(g) = s(g)((g^{-1}\tilde{x})\vert_{M})
= s(h_g k_g) ((k_g^{-1}  \tilde{x})\vert_{M})=s(k_g)
((k_g^{-1} \tilde{x})\vert_{M})
\end{align}
\par 
Hence, \eqref{e:lacol-map} and  \eqref{e:lacol-map-2} imply that $u=\tilde{v}\vert_{KM}$ and the claim is proved. 
\par 
We can thus compute  that: 
\begin{align*}
    \Psi_{K,t} ( \Psi_{K,s}(x)) & = 
    \Psi_{K,t} (v)  & (\text{by }\eqref{e:local-periodic-proof-1})
    \\ 
    & = \sigma_t (\tilde{v})\vert_{K} & (\text{by }\eqref{e:induced-local-maps-proof-1-2-3}) 
    \\ & = f^+_{K, t\vert_K}
    \left(\tilde{v}\vert_{KM} \right) 
    & (\text{by }\eqref{e:induced-local-maps-proof})
    \\
    & =  f^+_{K, t\vert_K}
    (u )
    \\& = f^+_{K, t\vert_K}
    \left(f^+_{KM, s\vert_{KM}} (\tilde{x}\vert_{KM^2})\right) \\
    & = \sigma_t (\sigma_s(\tilde{x}))\vert_K \\
    & = \tilde{x}\vert_K  & (\text{as } \sigma_c \circ \sigma_s = \Id)\\
    & = x. 
\end{align*}
\par 
Therefore, the proof of the lemma is complete. 
\end{proof}

\section{Continuity and closedness property of ANUCA}
\label{s:continuity-closedness}

We first prove that every ANUCA with finite memory is continuous. 
\begin{lemma}
\label{l:asyn-ca-continuous} Let $M$  be a finite subset of a group $G$. Let $A$ be a set and let $S=A^{A^M}$. Then for every $s \in S^G$, the ANUCA $\sigma_s \colon A^G \to A^G$ is continuous with  respect to the prodiscrete topology. 
\end{lemma}
\begin{proof}
It suffices to observe that for every finite subset $E \subset G$ and all $x, y \in A^G$ such that $x\vert_{EM}= y\vert_{EM}$, we have  $\sigma_s(x)\vert_E= \sigma_s(y) \vert_E$ and note that $EM\subset G$ is finite.  
\end{proof}
\par 
The following result  shows the continuity of ANUCA $\sigma_s$ with respect to the configuration $s \in S^G$. Note that we do not suppose the finiteness of neither the memory nor the alphabet.  
\begin{lemma}
\label{l:continuity} 
Let $M$  be a  subset of a group $G$. Let $A$ be a set and let $S=A^{A^M}$. Suppose that a sequence $(s_n)_{n \in \N}$ of elements of $S^G$ converges to some $s \in S^G$. Then for every $x \in A^G$, one has $\lim_{n \to \infty} \sigma_{s_n}(x) = \sigma_s(x)$.  
\end{lemma}

\begin{proof}
Let  $(s_n)_{n \in \N}$ be a sequence in  $S^G$ which converges to some $s \in S^G$. Let $x \in A^G$ and let $E\subset G$ be a finite subset. Then there exists $n_0 \in \N$ large enough so that for every $n \geq n_0$, we have $s_n\vert_{E}= s\vert_{E}$. From \eqref{e:induced-local-maps} and \eqref{e:induced-local-maps-proof}, we can compute: 
\begin{align*}
    \sigma_{s_n}(x)\vert_E  & = f_{E, (s_n)\vert_E}^+(x\vert_{EM}) \\ 
    & =  f_{E, s\vert_E}^+(x\vert_{EM}) \\ 
    & =  \sigma_{s}(x)\vert_E.
\end{align*}
\par 
We conclude that $\lim_{n \to \infty} \sigma_{s_n}(x) = \sigma_s(x)$ and the proof is complete. 
\end{proof}
\par 
Similarly, we can prove the following continuity of families of ANUCA with the same finite memory. 

\begin{lemma}
\label{l:continuity-x} 
Let $M$  be a finite   subset of a group $G$. Let $A$ be a set and let $S=A^{A^M}$. Suppose that
$\lim_{n \to \infty} x_n = x$ and $\lim_{n \to \infty} s_n = x$ for some $x \in A^G$, $s \in S^G$ and sequences $(x_n)_{n \in \N}$, $(s_n)_{n \in \N}$ in $A^G$ and $S^G$ respectively. 
Then one has $\lim_{n \to \infty} \sigma_{s_n}(x_n) = \sigma_s(x)$.   
\end{lemma}
\par 
\begin{proof}
 Let $s \in S^G$ and let $E\subset G$ be a finite subset. Since $M$ is  finite by hypothesis and since $\lim_{n \to \infty} x_n = x$, $\lim_{n \to \infty} s_n = s$ we can find $n_0 \in \N$ large enough so that for every $n \geq n_0$, we have $x_n\vert_{EM}= x\vert_{EM}$ and moreover 
$s_n\vert_{E}= s\vert_{E}$. 
We can thus compute again from the formula \eqref{e:induced-local-maps} and \eqref{e:induced-local-maps-proof}: 
\begin{align*}
    \sigma_{s_n}(x_n)\vert_E  & = f_{E, (s_n)\vert_E}^+(x_n\vert_{EM}) \\ 
    & =  f_{E, s\vert_E}^+(x\vert_{EM}) \\ 
    & =  \sigma_{s}(x)\vert_E.
\end{align*}
\par 
We conclude that $\lim_{n \to \infty} \sigma_{s_n}(x) = \sigma_s(x)$ and the proof is complete. 
\end{proof}

\par 
Note that when $G$ is a countable group, the topological space $A^G$ is metrizable for every set $A$. In fact, we  can define in this case the compatible \emph{Hamming metric} $d$  on $A^G$ associated with any given  exhaustion $(E_n)_{n \in \N}$ of finite subsets of $G$ by setting for all $x, y \in A^G$: 
\begin{equation}
    \label{e:hamming-metric}
    d(x,y)= 2^{-n(x,y)}, \quad \text{ where } n(x,y)= \sup \{k \in \N \colon x\vert_{E_k} = y \vert_{E_k}\}. 
\end{equation}
\par 
As for classical CA, we can show that the important closed image property also holds for ANUCA.  

\begin{theorem}
\label{t:closed-image-asynchronous} 
Let $M$  be a finite subset of a countable group $G$. Let $A$ be a finite set and let $S=A^{A^M}$. Then for every $s \in S^G$, the image $\sigma_s(A^G)$ is closed in $A^G$ with respect to the prodiscrete topology. 
\end{theorem}

\begin{proof}
First observe that $A^G$ is a compact metrizable space for the prodiscrete topology  by Tychonoff's theorem as $A$ is finite. 
Moreover, as $G$ is countable, the topological space $A^G$ is also metrizable with the standard Hamming metric (cf. \eqref{e:hamming-metric}). On the other hand, we know that $\sigma_s$ is continuous by Lemma~\ref{l:asyn-ca-continuous}. 
Consequently, $\sigma_s$ is a closed map and it follows in particular that $\sigma_s(A^G)$ is a closed subset of $A^G$. The proof is thus complete. 
\end{proof}
\par 
As an application, we obtain the following result:
\begin{corollary}
\label{l:post-surj-continuous-surj} Let $M$  be a finite subset of a countable group $G$. Let $A$ be a finite set and let $S=A^{A^M}$. Suppose that $s \in S^G$ and $\sigma_s$ is post-surjective. Then $\sigma_s$ is also surjective. 
\end{corollary}

\begin{proof}
By Theorem~\ref{t:closed-image-asynchronous}, it suffices to prove that the image $\sigma_s(A^G)$ is dense in $A^G$ with respect to the prodiscrete topopolgy. 
\par 
Let us fix an arbitrary configuration $x_0 \in A^G$ and let $y_0= \sigma_s(x_0) \in A^G$. Then by the definition of post-surjectivity, we deduce that every configuration that is asymptotic to $y_0$ must belong to the image of $\sigma_s$. On the other hand, the set of all configurations asymptotic to $y_0$ is clearly a dense subset of $A^G$. We deduce that $\sigma_s(A^G)$ contains a dense subset of $A^G$ and the conclusion follows. 
\end{proof}

\section{Stably injective ANUCA}
\label{s:stably-injective}
In general, ANUCA do not enjoy the $G$-equivariance property of the smaller class of CA but they satisfy the following fundamental property.  
 
\begin{lemma}
\label{l:pseudo-equivarian}
Let $M$ be a subset of a group $G$. Let $A$ be a set and let $S=A^{A^M}$. Let $s \in S^G$ then for all $g \in G$, we have 
\begin{equation}
\label{e:soft-equivariant}
\sigma_{gs}(gx) = g\sigma_s(x), \quad \text{ for every } x \in A^G. 
\end{equation} 
Moreover, if $\sigma_s$ is injective (resp. surjective, resp. pre-injective, resp. stably injective, resp. post-surjective) then so is $\sigma_{gs}$ for every $g \in G$. 
\end{lemma}

\begin{proof}
 Indeed, let $s \in S^G$, $g \in G$, $x \in A^G$ and $h \in G$, then we find from Definition \ref{d:most-general-def-asyn-ca}  that 
\begin{align*}
\sigma_{gs}(gx)(h) & =   gs(h)((h^{-1}gx)\vert_M)  \\
& = s(g^{-1}h) ((g^{-1}h)^{-1}x)\vert_M) \\
& = \sigma_s(x)(g^{-1}h)
\\ & = g\sigma_s(x)(h).  
\end{align*}
\par 
Consequently, we obtain $\sigma_{gs}(gx)= g\sigma_s(x)$ and the formula \eqref{e:soft-equivariant} is proved. 
\par 
Now fix $s \in S^G$. Suppose first that $\sigma_{s}$ is injective. Let $g \in G$ and let $x,y \in A^G$ be such that $\sigma_{gs}(x)= \sigma_{gs}(y)$. Then it follows from \eqref{e:soft-equivariant} that $g\sigma_s(g^{-1}x)= g\sigma_s(g^{-1}y)$. We deduce that $\sigma_s(g^{-1}x)= \sigma_s(g^{-1}y)$. Therefore, as  $\sigma_s$ is injective, we obtain $g^{-1}x=g^{-1}y$ and consequently $x=y$. Hence, $\sigma_{gs}$ is also injective for every $g \in G$. 
\par 
Similarly, suppose that $\sigma_s$ is surjective. Let $g \in G$ and let $y \in A^G$. Then we can find $x \in A^G$ such that $\sigma_s(x)=g^{-1}y$. It follows from 
\eqref{e:soft-equivariant} that $ \sigma_{gs}(gx) = g \sigma_s(x)=gg^{-1}y=y$. Since $y\in A^G$ is arbitrary, we conclude that $\sigma_{gs}$ is surjective as well.
\par 
Assume now that $\sigma_s$ is pre-injective. Let $g \in G$ and let  $x, y \in A^G$ be two asymptotic configurations such that $\sigma_{gs}(x)= \sigma_{gs}(y)$. Then we infer from \eqref{e:soft-equivariant} that 
$g\sigma_s(g^{-1}x)=  g\sigma_s(g^{-1}y)$ and therefore 
$\sigma_s(g^{-1}x)=\sigma_s(g^{-1}y)$. Since $x,y$ are asymptotic, so are $g^{-1}x$ and $g^{-1}y$. Thus, the pre-injectivity of $\sigma_s$ implies that $g^{-1}x=g^{-1}y$. Consequently, $x=y$ and we deduce that $\sigma_{gs}$ is pre-injective for every $g \in G$. 
\par 
Next, we suppose that $\sigma_s$ is post-surjective. Let $x , y, t\in A^G$ such that $y=\sigma_{gs}(x)$ and $y,t$ are asymptotic. Then \eqref{e:soft-equivariant} implies that $y=g\sigma_s(g^{-1}x)$ and thus $g^{-1}y=\sigma_s(g^{-1}x)$. As $\sigma_s$ is post-surjective and since $g^{-1}y$, $g^{-1}t$ are  asymptotic, we deduce that there exists a configuration $z \in A^G$ which is asymptotic to $g^{-1}x$ and $g^{-1}t= \sigma_s(z)$. Again, we infer from  \eqref{e:soft-equivariant} that 
$t=g\sigma_s(z)=\sigma_{gs}(gz)$. Note that $gz$ is asymptotic to $gg^{-1}x=x$. Therefore, we find that $\sigma_{gs}$ is also post-surjective for all $g \in G$. 
\par 
Finally, suppose that $\sigma_s$ is stably injective. Fix $ g \in G$  and observe that: 
\begin{align*}
\Sigma(gs) 
& =  
\overline{\{hgs \colon h \in G\}}\\
& = \overline{\{hs \colon h \in G\}} \quad \quad \quad (\text{since }G \text{ is a group}) \\ 
& = \Sigma(s)  \subset S^G. 
\end{align*}
\par 
Since $\sigma_s$ is stably injective, we deduce from the definition that $\sigma_{p}$ is injective for every $p \in \Sigma(s)= \Sigma(gs)$. We can thus  conclude that $\sigma_{gs}$ is stably injective for all $g \in G$. The proof of the lemma is complete. 
\end{proof} 
\par 
Lemma~\ref{l:pseudo-equivarian} implies in particular that  an injective ANUCA $\sigma_s$ is stably injective if and only if $\sigma_p$ is also injective for every \emph{limit point}  $p\in \Sigma(s)$, which justifies our choice of the terminology stable injectivity. We shall see later in Theorem~\ref{t:reversible-finite-alpha-beta},   Theorem~\ref{c:application-disturbed-ca-amenable}, Theorem~\ref{t:direct-post-injective} more stable properties of stably injective ANUCA. 

\par 
The next lemma allows us to  improve the statement concerning the stable injectivity of  Lemma~\ref{l:pseudo-equivarian}. 
\begin{lemma}
\label{l:post-injective-stable}
Let $M$  be a  subset of a group $G$. Let $A$ be a set and let $S=A^{A^M}$. Let $s \in S^G$. Then for every $p \in \Sigma(s)$ we have $\Sigma(p) \subset \Sigma(s)$. In particular,  $\sigma_s$ is  stably injective if and only if so is $\sigma_{p}$ for every $p \in \Sigma(s)$. 
\end{lemma}
\begin{proof}
Let $p \in \Sigma(s)$. Then for every $g \in G$, we have $gp \in \Sigma (s)$ as  $\Sigma(s)$ is a $G$-invariant subset of $S^G$. Since $\Sigma(s)$ is closed and $\Sigma(p) = \overline{\{gs \colon g \in G\}}$ by definition, we deduce that $\Sigma (p) \subset \Sigma(s)$ for all $p \in \Sigma(s)$.  
\par 
From this, the last statement follows immediately from the definition of stable injectivity and we simply note that $s \in \Sigma(s)$.  
\end{proof}
\par 
\begin{remark}
With the notation as in Lemma~\ref{l:post-injective-stable}, we remark that the inclusion $\Sigma(s) \subset \Sigma(p)$ may fail for some $p \in \Sigma(s)$. 
\par 
For example, let $G= \Z$, $A=\{0,1\}$,  and let $M= \{-1,0,1\} \subset G$. Let $S=A^{A^M}$ and let $u, v \colon A^M \to A$ be two distinct maps. Let $s \in S^G$ defined by $s(0)= u$ and $s(n)=v$ for all $n \in \Z \setminus \{0\}$. Then $\Sigma(s)$ consists of the translates of $s$ and the configuration constant $p\in S^G$ given by $p(n)= v$ for all $n \in \Z$. It is  clear that $\Sigma(p) =\{p\} \subsetneq \Sigma(s)$. 
\end{remark}

\section{The monoid of ANUCA} 
\label{s:monoid}

We begin with a lemma which describes the action of the translations of the configurations of local defining maps in a composition of ANUCA. 

\begin{lemma}
\label{l:inverse-asynchronous-ca} Let $M$  be a subset of a group $G$. Let $A$ be a set and let $S=A^{A^M}$. Let $p,q,s \in S^G$  and suppose that $\sigma_p\circ \sigma_q= \sigma_s$. Then for every $g \in G$, we have $\sigma_{gs} \circ \sigma_{gt} = \sigma_{gq}$. 
In particular, if 
$\sigma_p \circ \sigma_q= \Id_{A^G}$, then for every $g \in G$, one has $\sigma_{gp} \circ \sigma_{gq}=  \Id_{A^G}$. 
\end{lemma} 

\begin{proof}
Suppose first that $\sigma_p\circ \sigma_q= \sigma_s$ for some  $p,q,s \in S^G$. Let $g \in S$ and let $x \in A$. Then we infer from the formula \eqref{e:soft-equivariant} the following computation:  
\begin{align*}
    ( \sigma_{gp} \circ \sigma_{gq})(gx) 
    & = \sigma_{gp}( \sigma_{gq}(gx)) & \\ 
    & = \sigma_{gp}(g\sigma_q(x))  & \text{(by }\eqref{e:soft-equivariant}) 
    \\ & = 
  g\sigma_p(\sigma_q(x))  & \text{(by }\eqref{e:soft-equivariant}) 
  \\ & = g\sigma_s(x)  & \text{(as }\sigma_p \circ \sigma_s= \Id_{A^G}) 
  \\ & = \sigma_{gs}(gx). & \text{(by }\eqref{e:soft-equivariant}) 
\end{align*}
\par 
Since $y= g^{-1}x \in A^G$ is arbitrary, we deduce that 
\[ 
( \sigma_{gp} \circ \sigma_{gs})(y) = y 
\]
for all  $y \in A^G$ and the proof of the first statement is thus complete. The last statement is an obvious consequence. 
\end{proof}
 
The following result shows that the composition of two ANUCA is again an ANUCA. It follows that that set of ANUCA over a given universe and a given  alphabet form a monoid with respect to the composition operation. 

\begin{theorem}
\label{t:monoid-asynchronous}
Let $M, N \subset G$ be subsets of a group $G$. Let $A$ be a set and let $s \in S^G$,  $t \in T^G$ where  $S=A^{A^M}$, $T= A^{A^N}$. Then there exists $q \in Q^G$ where $Q= A^{A^{MN}}$ such that 
$\sigma_s \circ \sigma_t = \sigma_q$. 
\end{theorem}
 
\begin{proof}
Fix $g \in G$ and $x \in A^G$. Let us consider the induced local maps  $f_{\{g\}, s\vert_{\{g\}}}\colon A^{gM} \to A^{\{g\}} $ and $f_{gMN, t\vert_{gMN}}^+\colon A^{gMN} \to A^{gM}$ defined in Section~\ref{s:induced-local-map}. Then we infer from the formula \eqref{e:induced-local-maps} and  \eqref{e:induced-local-maps-general} the following computation: 
\begin{align*}
    \sigma_s(\sigma_t(x))(g) & = f_{\{g\}, s\vert_{\{g\}}}(f_{gMN, t\vert_{gMN}}^+ (x\vert_{gMN})) \\ 
    & =  s(g)( f_{MN, g^{-1}t\vert_{MN}}^+((g^{-1}x)\vert_{MN}))
\end{align*}
\par 
Therefore, if we define $q \in Q^G$ by setting $q(g)  \in A^{A^{MN}}$  to be the map 
\[ 
q(g) \coloneqq   s(g) \circ f_{MN, g^{-1}t\vert_{MN}}^+ 
\] for every $g \in G$ then it follows immediately  that $\sigma_s \circ \sigma_t= \sigma_q$. The proof of the theorem is thus complete. 
\end{proof}

\section{Reversibility of stably injective ANUCA} 
\label{s:reversible-stablyinjective}

Given a group $G$ and a set $A$, we say that an ANUCA $\sigma \colon A^G \to A^G$ is \textit{reversible} if it is injective and there exists an ANUCA with \textit{finite} memory $\tau \colon A^G \to A^G$ such that $\tau \circ \sigma= \Id$. Note that in our definitions, invertibility implies reversibility for ANUCA.  
We establish the following reversibility result for stably injective ANUCA. 

\begin{theorem} 
\label{t:reversible-finite-alpha}
Let $M$  be a finite subset of a countable group $G$. Let $A$ be a finite set and let $S=A^{A^M}$. Suppose that $\sigma_s$ is stably injective for some  $s \in S^G$. Then $\sigma_s$ is reversible. 
\end{theorem}

\begin{proof} 
We can suppose without loss of generality that $1_G \in M$. 
Since $G$ is countable by hypothesis, we can find an increasing sequence of finite subsets $(E_n)_{n \in \N}$ of $G$ such that $G= \cup_{n \in \N} E_n$ and $M \subset E_0$. 
\par 
For every $n \in \N$ and every configuration $w \in S^{E_n}$,  
we have an induced local map  $f_{E_n,s}^+ \colon A^{E_nM} \to A^{E_n}$ defined as in \eqref{e:induced-local-maps}. 
\par 
Let $\Gamma= \sigma_s(A^G)$. We claim that there exists $N\subset G$ finite and such that 
 $\sigma_s^{-1}(y)(g)\in A$ depends uniquely on the restriction $y\vert_{gN}$ for every configuration $y\in \Gamma$ and every group element $g \in G$. 
\par 
Indeed, suppose on the contrary that the claim is false. Then for every $n \in \N$, we can find $g_n \in G$ and $u_n, v_n \in A^G$ such that $u_n(g_n) \neq v_n(g_n)$ and $\sigma_s(u_n)\vert_{gE_n} = \sigma_s(v_n)\vert_{gE_n}$. Therefore, 
$g^{-1}\sigma_s(u_n)\vert_{E_n} = g^{-1}\sigma_s(v_n)\vert_{E_n}$ 
and 
we infer from \eqref{e:soft-equivariant} that 
\[
\sigma_{g_n^{-1}s}(g_n^{-1}u_n)\vert_{E_n} =g_n^{-1}\sigma_s(u_n)\vert_{E_n} =g_n^{-1}\sigma_s(v_n)\vert_{E_n}  =\sigma_{g_n^{-1}s}(g^{-1}v_n)\vert_{E_n}.
\]
\par 
Hence, by setting $ s_n= g_n^{-1}s\vert_{E_n} \in S^{E_n}$,  we can deduce from  \eqref{e:induced-local-maps-proof} that: 
\begin{equation}
\label{e:proof-reversible-1} 
f_{E_n, s_n }^+((g_n^{-1}u_n)\vert_{E_nM}) = f_{E_n,s_n }^+((g_n^{-1}v_n)\vert_{E_nM}). 
\end{equation}
\par 
Observe that $S$ is finite since $M$ and $A$ are finite. Hence, 
the space $S^G$ is compact with respect to the prodiscrete topology. Thus, the closed subset  $\Sigma(s) \subset S^G$ is also compact. 
\par 
Consequently, since $g_n^{-1}s \in \Sigma(s)$ we can find a subsequence $(g_{n_k}^{-1}s)_{k \in \N}$ of $(g_n^{-1}s)_{n \in \N}$ which converges to a configuration $p \in \Sigma(s)$. Hence, up to restricting again to another subsequence and reindexing, we can suppose without loss of generality that for every $k \in \N$, we have
\begin{equation}
    \label{e:proof-reversible-1-2a}
    g_{n_{k}}^{-1}s\vert_{E_k}= p \vert_{E_k}. 
\end{equation}
\par 
Note that $E_kM \subset E_{n_k}M$ for all $k \in \N$ since $n_k \geq k$.  Therefore, if we denote $x_k =g_{n_k}^{-1}u_{n_k}\in A^{G}$ and $y_k=g_{n_k}^{-1}v_{n_k}\in A^{G}$, we deduce immediately from \eqref{e:proof-reversible-1} and \eqref{e:proof-reversible-1-2a} that for every $k \in \N$, we have: 
\begin{equation}
    \label{e:proof-reversible-2} 
f_{E_k, p\vert_{E_k} }^+(x_k\vert_{E_kM}) = f_{E_n, p\vert_{E_k} }^+(y_k\vert_{E_kM}).  
\end{equation}
\par 
Consequently, the combination of the  relations  \eqref{e:induced-local-maps-proof} and \eqref{e:proof-reversible-2} 
imply that 
\begin{equation}
    \label{e:proof-reversible-2-3-4} 
\sigma_p(x_k) \vert_{E_k} =  \sigma_p(y_k) \vert_{E_k}.  
\end{equation}
\par 
Moreover, $x_k(1_G) \neq y_k(1_G)$ for all $k \in \N$ since $u_{n_k}(g_{n_k}) \neq v_{n_k}(g_{n_k})$. As the space $A^G \times A^G$ is compact with respect to the prodiscrete topology, we can suppose without loss of generality, up to passing to a subsequence, that $x_k$ converges to some $x \in A^G$ and $y_k$ converges to some $y\in A^G$ as well.   
\par 
Therefore, since $\sigma_p$ is continuous by Lemma~\ref{l:asyn-ca-continuous} and since $(E_k)_{k \in \N}$ is an exhaustion of $G$, we can infer from 
\eqref{e:proof-reversible-2-3-4} by passing to the limit that 
\[
\sigma_p(x)= \sigma_p(y). 
\]
\par 
On the other hand, we have $x(1_G) \neq y(1_G)$ since $x_k(1_G) \neq y_k(1_G)$ for all $k \in \N$. In particular, $x \neq y$ and it follows that $\sigma_p$ is not injective. However, since $\sigma_s$ is stably injective and $p \in \Sigma(s)$, we deduce that $\sigma_p$ is injective, which is a contradiction.
\par 
Hence, we have proved the claim that 
there exists a finite subset $N\subset G$  such that 
 $\sigma_s^{-1}(y)(g)\in A$ depends uniquely on the restriction $y\vert_{gN}$ for every configuration $y\in \Gamma$ and every group element $g \in G$. 
 \par 
To complete the proof of the theorem, let $T= A^{A^N}$. We construct a configuration $q \in T^G$ as follows. Fix some $a_0 \in A$. For every $g \in G$, the property of the set $N$ shows that we have a well-defined map $\varphi_g \colon \Gamma_{gN} \to A$ given by the formula: 
\[
\varphi_g(z) = \sigma_s^{-1}(y)(g)  
\]
for every $z \in \Gamma_{gN}$ and $y \in \Gamma$ which extends $z$. We define $q(g) \colon A^N \to A$ by setting 
$q(g)(t)= \varphi_g(gt)$ for all $t \in A^N$ such that $gt \in \Gamma_{gN}$ and we simply put $q(g)(t)=a_0$ whenever $t \in A^N$ such that $gt \notin \Gamma_{gN}$. 
\par 
Hence, we obtain $q \in T^G$ and it is clear from our construction that $\sigma_s^{-1}(y)= \sigma_q(y)$ for all $y \in \Gamma$. The proof is thus complete. 
\end{proof} 

\par 
Under the same assumptions of Theorem~\ref{t:reversible-finite-alpha}, we can actually show that stable injectivity implies \emph{stable reversibility} which is a stronger property than reversibility. 

\begin{definition}
\label{d:stable-reversibility}
Let $M$  be a finite subset of a countable group $G$. Let $A$ be a finite set and let $S=A^{A^M}$. Let $s \in S^G$. Then $\sigma_s$ is said to be \emph{stably reversible} if there exist $N\subset G$ finite and $q \in T^G$ where $T=A^{A^N}$ such that for every $p \in \Sigma(s)$, we can find $w \in \Sigma(q)$ such that $\sigma_w \circ \sigma_p=\Id$. 
\end{definition}
\par 
Our result can be stated as follows. 

\begin{theorem}
\label{t:reversible-finite-alpha-beta}
Let $M$  be a finite subset of a countable group $G$. Let $A$ be a finite set and let $S=A^{A^M}$. Let $s \in S^G$ and suppose that $\sigma_s$ is stably injective.  
Then $\sigma_s$ is stably reversible. 
\end{theorem}

\begin{proof}
By Theorem~\ref{t:reversible-finite-alpha}, there exists a finite subset $N \subset G$ and $q \in T^G$ where $T=A^{A^N}$ such that $\sigma_q \circ \sigma_s= \Id$. 
\par 
Let $p \in \Sigma(s)$. Then there exists a sequence $(g_n)_{n \in \N}$ in  $G$ such that the sequence $(g_ns)_{n \in \N}$ converges to $p$ in the space $S^G$ with respect to the prodiscrete topology. 
\par 
Since $T^G$ is compact with respect to the prodiscrete topology as $T$ is finite, we can suppose, up to passing to a subsequence, that  $(g_nq)_{n \in \N}$ converges to some  $w \in T^G$. In particular, we have $w \in \Sigma(q)$ by the definition of $\Sigma(q)$. 
\par 
We are going to prove that $\sigma_w \circ \sigma_p= \Id$. It suffices to show that $\sigma_w(y)=x$ for all $y= \sigma_p(x)$ where $x \in A^G$. 
\par 
Hence, let us fix $x \in A^G$ and let $y=\sigma_p(x)$. Since  $\lim_{n \to \infty}g_n q = w$,  Lemma~\ref{l:continuity} implies that $\lim_{n \to \infty}\sigma_{g_n p}(y) = \sigma_w(y)$. 
\par 
We claim that $\lim_{n \to \infty}\sigma_{g_n q}(\sigma_{p}(x)) = x$. Indeed, let $F \subset G$ be a finite subset. As $\lim_{n \to \infty} g_n s = p$, there exists $n_0 \in \N$ such that $p\vert_{FN}= (g_ns)\vert_{FN}$ for all $n \geq n_0$. 
\par 
From  \eqref{e:induced-local-maps} and  \eqref{e:induced-local-maps-general}, we find that  for all $n \geq n_0$:
\begin{align*}
    \sigma_{g_n q}(\sigma_{p}(x))\vert_F & =
    f_{F, (g_nq)\vert_{F}}^+
    (f_{FN, p\vert_{FN}}^+ (x\vert_{FNM})) \\
    & = f_{F, (g_nq)\vert_{F}}^+
    (f_{FN, (g_ns)\vert_{FN}}^+ (x\vert_{FNM})) \\ 
    & =  \sigma_{g_n q}(\sigma_{g_ns}(x))\vert_F \\ 
    & = x\vert_F. 
\end{align*}
\par 
It follows that $\lim_{n \to \infty}\sigma_{g_nq}(y)=\sigma_{g_n q}(\sigma_{p}(x)) = x$ and the claim is proved. 
\par 
Consequently, $x= \sigma_w(y)$ as we als have $\lim_{n \to \infty}\sigma_{g_n p}(y) = \sigma_w(y)$. Therefore, $\sigma_w \circ \sigma_p=\Id$ as $x \in A^G$ is arbitrary. We conclude that $\sigma_s$ is stably reversible and the proof of the theorem is thus complete. 
\end{proof}
\par

\section{Disturbance of CA over   residually finite group universes} 
\label{s:app-disturbed-sofic}

Recall that a configuration $s \in S^G$, where $S$ is a set and $G$ is a group, 
 is said to be \emph{constant} if $s(g)=s(h)$ for all $g,h \in G$. We have the following simple observation.

\begin{lemma}
\label{l:asym-constant-s}
Let $S$ be a finite set and let $G$ be a group. Suppose that $s \in S^G$ is asymptotic to a constant configuration $c \in S^G$. Then we have: 
\begin{equation}
    \label{e:application-disturbed-ca}
    \Sigma(s) = \{gs \colon g \in G\} \cup \{c\} \subset S^G.
\end{equation}
\end{lemma}

\begin{proof} 
For this, let $S$ be a finite symmetric generating set of $G$ and let $B_{S}(r) \subset G$ be the ball of radius $r$ in the connected Cayley graph $C_S(G)$ of $G$ associated with $S$ and the corresponding metric $d_S \colon G \times G \to \N$ defined as the length of shortest path in $C_S(G)$ that connects two vertices. 
\par 
Since $s$ is asymptotic to the constant configuration  $c \in S^G$, we can find $r_0 \geq 1$ such that $s(g)=c(1_G)$ for all $g \in G \setminus B_S(r_0)$. Let $(g_n)_{n \in \N}$ be an arbitrary sequence of elements in $G$ such that $d_S(g_n,1_G)=n+r_0+1$ for every $n \in \N$. Then it is clear from the triangle inequality that $g_ns(g)=c(1_G)$ for all $g \in B_S(n)$. Consequently, we deduce that the sequence $(g_ns)_{n \in \N}$ converges to $c \in S^G$. It follows that $c \in \Sigma(s)$. 
\par 
It is clear from the definition that $gs \subset \Sigma(s)$ for all $g \in G$. 
Hence, we find that 
$\{gs \colon g \in G\} \cup \{c\} \subset \Sigma(s)$. 
\par 
Conversely, let $p \in \Sigma(s)$ and suppose that $p \neq gs$ for all $g \in G$. Then $p$ is the limit of $(g_ns)_{n \in \N}$ for some sequence $(g_n)_{n \in \N}$ of distinct elements of $G$. 
\par 
Since the ball $B_S(r)$ is finite for every $r \in \N$, we can thus find a subsequence $(g_{n_k})_{k \in \N}$ of $(g_n)_{n \in \N}$ such that $(d_S(g_n, 1_G))_{n \in \N}$ forms a strictly increasing sequence of positive integers. In particular, we deduce that $d_S(g_n, 1_G) \geq n$ for every $n \in \N$. 
\par 
It is then clear that $(g_{n+r_0+1}s)\vert_{B_S(n)} = c\vert_{B_S(n)}$ for all $n \in \N$. It follows that $p\vert_{B_S(n)}=c\vert_{B_S(n)}$ for all $n \in \N$. Consequently, we have $p=c$. This shows  that 
$\Sigma(s) \subset \{gs \colon g \in G\} \cup \{c\}$ and the relation \eqref{e:application-disturbed-ca} is proved. 
The proof of the lemma is thus complete. 
\end{proof} 
\par 
As an application of Theorem~\ref{t:reversible-finite-alpha-beta}, we obtain the following  surjunctivity property of locally disturbed injective classical CA.

\begin{theorem}
\label{c:application-disturbed-ca} Let $M$ be a finite subset of a residually finite group $G$. Let $A$ be a finite set and let $S=A^{A^M}$. Let $s \in S^G$ be asymptotic to a constant configuration $c \in S^G$. Then $\sigma_s$ and $\sigma_c$ are invertible whenever they are both  injective. 
\end{theorem}

\begin{proof} 
As residually finite groups are  sofic,  Gromov-Weiss the surjunctivity theorem for CA  implies that $\sigma_c$ is surjective and thus invertible. 
\par 
Let $\Gamma = \sigma_s(A^G)$ and  suppose on the contrary that $\sigma_s$ is not surjective. As $\sigma_s(A^G)$ is closed in $A^G$ in the prodiscrete topology by Theorem~\ref{t:closed-image-asynchronous}, there exists a finite subset $\Omega \subset G$ such that $\Gamma_{\Omega} = f^+_{\Omega, s\vert_\Omega} (A^{\Omega M}) \subsetneq A^{\Omega}$.
\par
Since $\sigma_s$ and $\sigma_c$ are injective, we deduce from Lemma~\ref{l:asym-constant-s} and Lemma~\ref{l:pseudo-equivarian} that $\sigma_s$ is stably injective. It follows that $\sigma_s$ is reversible by Theorem~\ref{t:reversible-finite-alpha}. Consequently, up to enlarging $M$ without loss of generality, we can find $t \in S^G$ such that $\sigma_t \circ \sigma_s = \Id$. 
\par 
Up to enlarging $M$ again, we can suppose without loss of generality that $1_G\in M$ and $M$ is symmetric, i.e., $M=M^{-1}$.  
Since $s$ and $c$ are asymptotic, we can find a finite subset $E \subset G$  such that $M \cup \Omega \subset E$ and $s \vert_{G \setminus E} = c \vert_{G \setminus E}$. 
\par 
Since $G$ is residually finite, Lemma~\ref{l:separation-forte-res} provides a finite group $H$ and a surjective group homomorphism $\varphi \colon G \to H$ such that the restriction map $\varphi\vert_{E} \colon EM^2 \to H$ is injective and there exists a finite subset $K \subset G$   such that $E \subset K$ and  $\varphi\vert_{K} \colon K \to H$ is  a bijection and $\varphi(KM \setminus K) \cap \varphi(E)=\varnothing$. 
\par 
The configurations $s,t$ induce
the maps $\Psi_{K,s}$ and  $\Psi_{K,t} \colon A^K \to A^K$ defined as in 
Section~\ref{s:induced-local-map-2}.  Every $x \in A^K$ defines $\tilde{x} \in A^G$ by $\tilde{x}(g)=x(k_g)$ for all $g \in G$ and the unique $k_g \in K$ such that $\varphi(g) = \varphi (k_g)$.
Then 
\begin{align}
    \label{e:local-periodic-proof-1-a}
\Psi_{K,s}(x) \coloneqq \sigma_s(\tilde{x})\vert_K, \quad  
\Psi_{K,t}(x) \coloneqq \sigma_t(\tilde{x})\vert_K.
\end{align}    
\par 
Since  $\sigma_t \circ \sigma_s = \Id$ and  $s(g)=s(k_g)=c(0)$ for all $g \in KM \setminus K$, we infer from Lemma~\ref{l:induced-local-map-reversible-1} that $\Psi_{K,t}\circ \Psi_{K,s} = \Id$. 
\par 
As $A^K$ is finite, we deduce that $\Psi_{K,t}$ and $\Psi_{K,s}$ are bijections. In particular, since $\Omega \subset K$, it follows that
\[
\Gamma_\Omega = \sigma_s(A^G)\vert_{\Omega} \supset \{ \sigma_s(\tilde{x})\colon x \in A^K\}\vert_\Omega = (\im \Psi_{K,s})\vert_\Omega = (A^K)\vert_\Omega = A^{\Omega}. 
\]
\par 
Hence, we obtain a contradiction to the choice of $\Omega$. We conclude that $\sigma_s$ is surjective. Since $\sigma_t \circ \sigma_s = \Id$, it follows at once that $\sigma_s$ and $\sigma_t$ are invertible. The proof of the theorem is thus complete. 
\end{proof}

\par 
Therefore, we see that when an injective ANUCA $\sigma_{s}$ was obtained by disturbing an injective classical  CA, i.e., when the configuration  $s$ is asymptotic to a constant configuration, then $\sigma_s$ is in fact invertible if the universe is a residually finite group.

\section{Disturbance of CA over amenable group universes}
\label{s:app-disturbed-amenable}

When the universe is an amenable group, Theorem~\ref{c:application-disturbed-ca} can be strengthened as follows. In essence, what happens in this case is that because of the Garden of Eden theorem, one cannot obtain injective ANUCA by disturbing the local transition rules of a finite number of cells of non-injective CA.

\begin{theorem}
\label{c:application-disturbed-ca-amenable} Let $G$ be an amenable group and let $M \subset G$ be finite. Let $A$ be a finite set and let $S=A^{A^M}$. Suppose that $\sigma_s$ is injective for some $s \in S^G$ asymptotic to a constant configuration $c$. Then $\sigma_c$ and $\sigma_s$ are  invertible. 
\end{theorem}

\begin{proof}
We can suppose without loss of generality that $G$ is a finitely generated group up to restricting to the subgroup generated by the union of $M$ with the largest finite subset of $G$ over which $c$ is different from $s$. 
Since $c$ and $s$ are asymptotic, there exists $E \subset F$ finite such that $c\vert_{G \setminus E}= s\vert_{G \setminus E}$. 
\par 
First, we claim that $\sigma_c$ is surjective. 
 Indeed, suppose on the contrary that $\sigma_c$ is not surjective. Then we infer from the Garden of Eden theorem (cf.~\cite{ceccherini})  that $\sigma_c$ is not pre-injective. Consequently, we can find two distinct  asymptotic configurations $u,v \in A^G$ such that $\sigma_c(u)= \sigma_c(v)$. In particular, we can find $F \subset G$ finite such that  $u \vert_{G \setminus F} = v \vert_{G \setminus F}$. Up to replacing $u,v$ by a suitable translation, we can clearly suppose that $F \cap EM = \varnothing$. 
 \par 
It follows immediately from  the formula \eqref{e:induced-local-maps} and  \eqref{e:induced-local-maps-general} that $\sigma_s(u)= \sigma_s(v)$ and thus $\sigma_s$ is not injective.  The obtained contradiction proves  the claim that $\sigma_c$ is surjective. 
\par 
Now fix $z \in A^G$ and a finite subset $N \subset G$ containing $E$. Consider the set $V= \{x \in A^G \colon x\vert_{G \setminus N} = z\vert_{G \setminus N} \}$. Let $U= \sigma_c^{-1}(V) \subset A^G$. Since $\sigma_c$ is surjective, we have $\vert U \vert \geq \vert V \vert$. 
\par 
On the other hand, observe that $\sigma_s(U) \subset V$ as $s \vert_{G \setminus N} = c \vert_{G \setminus N}$ so that $\vert \sigma_s(U) \vert \leq \vert V \vert$. Hence, by combining with the inequality $\vert U \vert \geq \vert V \vert$  and the injectivity of $\sigma_s$, we find that
\begin{equation*} 
\label{e:amenable-group-stable-deformation}
 \vert U \vert = \vert \sigma_s(U) \vert \leq \vert V \vert \leq \vert U \vert. 
\end{equation*} 
\par
Therefore, $\vert \sigma_s(U) \vert = \vert V \vert$ and thus $\sigma_s(U)= V $ as $\sigma_s(U) \subset V$. 
\par 
Since $N$ is arbitrary, we deduce that the image $\sigma_s(A^G)$ is dense in $A^G$ with respect to the prodiscrete topology. We can thus conclude that $\sigma_s(A^G)= A^G$ since $\sigma_s(A^G)$ is closed in $A^G$ by Theorem  \ref{t:closed-image-asynchronous}.
\par 
Therefore, $\sigma_s$ is surjective and thus bijective. Since $\sigma_s$ is injective, we deduce from  Lemma~\ref{l:pseudo-equivarian}  that $\sigma_{gs}$ is also injective for every $g \in G$.   
\par 
Note also that $\vert \sigma_c^{-1}(V) \vert = \vert V \vert $ since they both equal to $ \vert U \vert=\vert \sigma_s(U) \vert$. Hence, as $z \in A^G$ is arbitrary, we deduce that $\sigma_c$ is injective and thus bijective. This proves the first part of the conclusion of the theorem.  
\par 
As a consequence, $\sigma_s$ is stably injective since we know by Lemma~\ref{l:asym-constant-s}   that $\Sigma(s)=\{gs \colon g \in G\} \cup \{c\}$. Therefore, we infer from Theorem~\ref{t:reversible-finite-alpha}  
that $\sigma_c$ and  $\sigma_s$ are reversible and thus invertible as they are surjective. Therefore,  there exist  $R \subset G$ finite and $t,d \in T^G$ where $T=A^{A^R}$ such that $\sigma_s^{-1}=\sigma_t$ and $\sigma_c^{-1}=\sigma_d$. Note that $d$ is constant since $c$ is constant. 
\par
The proof of the theorem is thus complete. 
\end{proof}

\section{Generalization to ANUCA of bounded singularity} 
\label{s:bounded-singularity-section}
When the universe is a free abelian  group, we can establish the following invertibility result of the large class of stably injective \emph{ANUCA of bounded singularity} that we describe below. 
\par 
Given $g, h \in \Z^d$ and a box $K= \prod_{j=1}^d\llbracket a_j,b_j \rrbracket^d \in \Z^d$ where $\llbracket a_j,b_j \rrbracket = \{a_j,  \dots, b_j\}$, we say that that 
$g \equiv h$ (mod $K$) if 
$g_j \equiv h_j$ (mod $a_j-b_j+1$) for every $j=1, \dots, d$ where $g=(g_1, \dots, g_d)$ and $h = (h_1, \dots, h_d)$. 
 \par 
Given subsets $M,K$ of a group $G$. We can define the \emph{$M$-interior}, the \emph{$M$-exterior}, and the \emph{$M$-boundary} of $K$ respectively by 
\begin{align*}
\partial_M^-K & \coloneqq \{ g \in K \colon gM \subset K \},\\ 
\partial_M^+K & \coloneqq KM \setminus K, \\
\partial_M K 
& 
\coloneqq 
\partial_M^+K \cup (K \setminus \partial_M^-K). 
\end{align*}
 
\begin{definition}
\label{d:bounded-singularity}
Let 
$ M \subset \Z^d$ ($d\in \N$)
be finite. Let $A$ be a finite set and let $S=A^{A^M}$. Given $s \in S^{\Z^d}$, we say that     $\sigma_s$ has \emph{bounded singularity} if for all finite subset $E \subset \Z^d$, there exists a box $K \subset \Z^d$ containing $E$ such that  
$s(g) = s(k_g)$ for all $g \in \partial_E^+K$ and the unique $k_g \in K$ with  $k_g \equiv g$ (mod $K$).  
\end{definition}
\par 
For example, 
let $M = \llbracket -r, r \rrbracket^2$ and let $p \geq 2r +1$. Let $A$ be a finite set and let $S= A^{A^M}$. Then $\sigma_s$ has bounded singularity for every $s \in S^{\Z^2}$ which is constant on $\Z^2 \setminus p\Z^2M$. 
The following example is more general.  
\begin{example}
For every $r \geq 0$, let $M_r = \llbracket -r, r \rrbracket^2$. Let $(K_n)_{n \geq 0}$ be a nested sequence of boxes such that $\Z^2 = \cup_{n \geq 0} K_n$. Let $A$ be a finite set and let $S= A^{A^{M_{r_0}}}$. Then the ANUCA  $\sigma_s$ has bounded singularity for every $s \in S^{\Z^2}$ such that $s$ is constant on each of $\partial_{M_n} K_n$. 
\end{example}

\par 
We can now prove the main result Theorem~\ref{t:intro-singularity} in the Introduction   whose proof is similar to the proof of Theorem~\ref{c:application-disturbed-ca}.

\begin{proof}[Proof of Theorem~\ref{t:intro-singularity}] 
Since $\sigma_s$ is stably  injective, we infer from  Theorem~\ref{t:reversible-finite-alpha}  that $\sigma_s$ is reversible. Hence, up to enlarging $M$ without loss of generality, we can find $t \in S^G$ such that $\sigma_t \circ \sigma_s = \Id$. 
\par 
Let $\Gamma = \sigma_s(A^{\Z^d})$ and  suppose on the contrary that $\sigma_s$ is not invertible. In particular, $\sigma_s$ is not surjective and we infer   Theorem~\ref{t:closed-image-asynchronous} that there exists a finite subset $E \subset \Z^d$ such that $M \subset E$ and $\Gamma_{E} = f^+_{E, s\vert_E} (A^{E M}) \subsetneq A^{E}$.
\par 
Since $\sigma_s$ has bounded singularity, we can find a box $K \subset \Z^d$ which contains $E$ and such that  
$s(g) = s(k_g)$ for all $g \in \partial_E^+K$ and the unique $k_g \in K$ with  $k_g \equiv g$ (mod $K$).  
\par 
The configurations $s,t$ induce
the maps $\Psi_{K,s}$ and  $\Psi_{K,t} \colon A^K \to A^K$ defined as in 
Section~\ref{s:induced-local-map-2}.  Every $x \in A^K$ defines $\tilde{x} \in A^{\Z^d}$ by $\tilde{x}(g)=x(k_g)$ for all $g \in \Z^d$ and the unique $k_g \in K$ such that $k_g \equiv g$ (mod $K$). 
Then we have: 
\begin{align}
    \label{e:local-periodic-proof-1-a}
\Psi_{K,s}(x) \coloneqq \sigma_s(\tilde{x})\vert_K, \quad  
\Psi_{K,t}(x) \coloneqq \sigma_t(\tilde{x})\vert_K.
\end{align}    
\par 
Since  $\sigma_t \circ \sigma_s = \Id$ and  $s(g)=s(k_g)$ for all $g \in KM \setminus K$ (as $M \subset E$), Lemma~\ref{l:induced-local-map-reversible-1} implies that $\Psi_{K,t}\circ \Psi_{K,s} = \Id$. It follows that $\Psi_{K,t}$ and $\Psi_{K,s}$ are bijective since   $A^K$ is finite. As $E \subset K$, we deduce that 
\[
\Gamma_E = \sigma_s(A^{\Z^d})\vert_{E} \supset \{ \sigma_s(\tilde{x})\colon x \in A^K\}\vert_E = (\im \Psi_{K,s})\vert_E = (A^K)\vert_E = A^{E}, 
\]
which contradicts the choice of $E$. We conclude that $\sigma_s$ is surjective and thus invertible since $\sigma_t \circ \sigma_s = \Id$. The proof of the theorem is complete. 
\end{proof}
\par 
Using Lemma~\ref{l:separation-forte-res} and Lemma~\ref{l:induced-local-map-reversible-1}, we see that  Definition~\ref{d:bounded-singularity} and Theorem~\ref{t:intro-singularity} can be easily generalized, \textit{mutatis mutandis}, to finitely generated group universes.

\section{Stable reversibility and direct finiteness of ANUCA} 
\label{s:direct-finintess}

\begin{theorem}
\label{t:direct-post-injective}
Let $M$  be a finite subset of a countable group $G$. Let $A$ be a finite set and let $S=A^{A^M}$. Let $s, t \in S^G$ and suppose that $\sigma_t\circ \sigma_s= \Id $. Then $\sigma_s$ is stably injective. Moreover, for every $p \in \Sigma(s)$, there exists $q \in \Sigma(t)$ such that $\sigma_q \circ \sigma_p= \Id$. In particular, $\sigma_s$ is stably reversible. 
\end{theorem}

\begin{proof}
Let $p \in \Sigma(s)$ then there exists a sequence $(g_n)_{n \in \N}$ of elements of $G$ such that $\lim_{n \to \infty} g_n s =p$. Since $M$ and $A$ are finite, $S$ is also finite and it follows that $S^G$ is compact by Tychonoff's theorem. Hence, up to passing to a subsequence, we can suppose without loss of generality that 
$\lim_{n \to \infty} g_n t=q$ for some $q \in S^G$. 
\par 
We claim that $\sigma_q \circ \sigma_p= \Id$. Indeed, let $x \in A^G$ and let $E \subset G$ be a finite subset. 
Since $\lim_{n \to \infty} g_n s =p$ and $\lim_{n \to \infty} g_n t=q$, we can find $n_0 \in \N$ large enough so that for all $n \geq n_0$, we have $(g_ns)\vert_{EM}= p \vert_{EM}$ and 
$(g_nt)\vert_E= q \vert_E$. 
\par 
Note that $\sigma_{g_nt}\circ \sigma_{g_ns}=\Id$ by Lemma~\ref{l:inverse-asynchronous-ca} as $\sigma_t \circ \sigma_s=\Id$ by hypothesis. Consequently, we infer from the formula \eqref{e:induced-local-maps} and  \eqref{e:induced-local-maps-general} that: 
\begin{align*} 
    \sigma_q( \sigma_p(x)) \vert_E  & =
    f_{E, q\vert_{E}}^+
    (f_{EM, p\vert_{EM}}^+ (x\vert_{EM^2})) \\
    & = f_{E, (g_nt)\vert_{E}}^+
    (f_{EN, (g_ns)\vert_{EM}}^+ (x\vert_{EM^2})) \\ 
    & =  \sigma_{g_n q}(\sigma_{g_ns}(x))\vert_E \\ 
    & = x\vert_E.  
\end{align*} 
\par 
Since $E$ is arbitrary, we deduce that $\sigma_q(\sigma_p(x))= x$ for all $x \in A^G$ and the claim is proved. Since clearly $q \in \Sigma(t)$, the last statement of the theorem is proved. In particular, we find that $\sigma_s$ is stably injective by definition and the proof is therefore complete. 
\end{proof}

\par 
Combining  Theorem~\ref{t:direct-post-injective} with Theorem~\ref{t:reversible-finite-alpha}, we can now give the proof of Theorem~\ref{t:intro-characterization-reversible-ANUCA} in the Introduction which gives various characterizations of the reversibility of ANUCA.  

\begin{proof}[Proof of Theorem~\ref{t:intro-characterization-reversible-ANUCA}] 
It is clear from the definition of stable reversibility that (ii)$\implies$(i).   Theorem~\ref{t:direct-post-injective} tells us that (i)$\implies$(iii). Finally, the implication (iii)$\implies$(ii) follows from Theorem~\ref{t:reversible-finite-alpha}. 
\end{proof}

\section{Pointwise uniform post-surjectivity}  \label{s:pointwise-uniform-post}

\par 
Given a map $\tau \colon A^G \to A^G$ where $G$ is a group and  $A$ is a set. Then $\tau$ is   \emph{pre-injective} if $\tau(x) = \tau(y)$ implies $x= y$ whenever $x, y \in A^G$ are asymptotic, and $\tau$ is \emph{post-surjective} if for all $x, y \in A^G$ with $y$ asymptotic to $\tau (x)$, we can find $z \in A^G$ asymptotic to $x$ such that  $ \tau(z)=y$. We shall see in Corollary~\ref{l:post-surj-continuous-surj} that every post-surjective ANUCA is automatically surjective. 
\par 
The dual-surjunctivity  version of Gottschalk's conjecture was introduced recently by 
Capobianco, Kari, and Taati in \cite{kari-post-surjective} and states that if $G$ is a group and $A$ is a finite set, then every  post-surjective CA must be  pre-injective. Moreover, the authors settled in the same paper \cite{kari-post-surjective} the dual-surjunctivity conjecture for CA over sofic universes. 
 \par 

We establish the following pointwise uniform  post-surjectivity of ANUCA. 

\begin{lemma} 
\label{l:uniform-psot-surjectivity}
Let $M$  be a finite subset of a  countable group $G$. Let $A$ be a finite set and let $S=A^{A^M}$. Suppose that $\sigma_s$ is post-surjective for some $s \in S^G$. Then for each $g \in G$, there exists a finite subset $E\subset G$ such that for all $x, y\in A^G$ with  $y\vert_{G \setminus \{g\}} =\sigma_s(x)\vert_{G \setminus \{g\}}$, there exists $z \in A^G$ such that $\sigma_s(z)=y$ and 
$z\vert_{G \setminus gE}= x\vert_{G \setminus gE}$. 
\end{lemma}

\begin{proof} 
To simplify the notation,  we will only treat the case when $g=1_G$ since  the general case is similar. 
\par 
Without loss of generality, we can clearly suppose that $1_G\in M$ and $M$ is symmetric, i.e., $M=M^{-1}$. 
\par 
Since $G$ is countable by hypothesis, we can find an increasing sequence $(E_n)_{n \in \N}$ of finite subsets of $G$ such that $G= \cup_{n \in \N}E_n$,  $1_G \in E_0$, and moreover for every $n \in \N$: 
\[
E_n M \subset E_{n+1}. 
\] 
\par 
We suppose on the contrary that there does not exist a  finite subset $E \subset G$ with the property described in the conclusion of the lemma. It follow that there exists for every $n \in \N$ two  configurations $x_n, y_n \in A^G$ with $y_n\vert_{G \setminus \{1_G\}} =\sigma_s(x_n)\vert_{G \setminus \{1_G\}}$ but for all $z \in A^G$ satisfying $z\vert_{G \setminus E_n}= x_n\vert_{G \setminus E_n}$, one has  $\sigma_s(z) \neq y_n$.
\par 
Note that we have 
$G \setminus E_{m} \subset G \setminus E_n$ for all $m \geq n \geq 0$. Hence, by the choice of $x_n$ and $y_n$, we find that for every $m \geq n \geq 0$, we have $\sigma_s(z) \neq y_m$ for every $z \in A^G$ with $z\vert_{G \setminus E_n}= x_m\vert_{G \setminus E_n}$.  
\par 
On the other hand, the space $A^G$ is compact by Tychonoff's theorem. Therefore, we can, up to passing to a subsequence, suppose without loss of generality that 
$\lim_{n \to \infty} x_n=x$ and $\lim_{n \to \infty}y_n=y$ for some $x,y \in A^G$.   
\par 
By passing the relation $y_n\vert_{G \setminus \{1_G\}} =\sigma_s(x_n)\vert_{G \setminus \{1_G\}}$ to the limit when $n$ goes to $\infty$, we deduce that 
\[ 
y\vert_{G \setminus \{1_G\}} =\sigma_s(x)\vert_{G \setminus \{1_G\}}  
\] 
since $\sigma_s$ is continuous by Lemma~\ref{l:continuity-x}. Hence, $y$ is asymptotic to $\sigma_s(x)$.
\par It follows from the post-surjectivity of $\sigma_s$ that there exists $w \in A^G$ such that $w$ is asymptotic to $x$ and $\sigma_s(w)= y$. 
In particular, we can find $k \in \N$ such that $w \vert_{G \setminus E_k}= x \vert_{G \setminus E_k}$. 
\par 
As $\lim_{n \to \infty} x_n=x$ and $\lim_{n \to \infty} y_n=y$, we can choose $m > k$ such that $x_m\vert_{E_{k}M^2}= x \vert_{E_{k}M^2}$ and $y_m\vert_{E_{k}M^2}= y \vert_{E_{k}M^2}$. 
\par 
Consequently, we have $w\vert_{E_{k}M^2 \setminus E_k} =x_m\vert_{E_{k}M^2 \setminus E_k}$. Therefore, we obtain a well-defined  configuration $z\in A^G$ by setting 
\begin{align}
\label{e:post-surjectivity-1-2-3}
z \vert_{E_{k}M^2}= w\vert_{E_{k}M^2}, \quad \quad 
z \vert_{G \setminus E_{k}}=x_m \vert_{G \setminus E_{k}}.
\end{align} 
\par 
Since $M$ is a symmetric memory set of $\sigma_s$ and $1_G \in E_kM$, we infer the formula  \eqref{e:induced-local-maps}-\eqref{e:induced-local-maps-general}, and the choice of $x_m$, $y_m$ that: 
\begin{align*}
    \sigma_s(z)\vert_{G \setminus E_kM} & =  \sigma_s(x_m)\vert_{G \setminus E_kM}   \\
    &  = y_m \vert_{G \setminus E_k M}. 
\end{align*}
\par 
Similarly, we deduce from the relations \eqref{e:induced-local-maps}-\eqref{e:induced-local-maps-general},  $y_m\vert_{E_{k}M^2}= y \vert_{E_{k}M^2}$, and $ \sigma_s(w)=y$ that: 
\begin{align*}
    \sigma_s(z)\vert_{ E_kM} & =  \sigma_s(w)\vert_{ E_kM}   \\
    &  = y\vert_{ E_kM}  \\
    & = y_m \vert_{ E_kM}. 
\end{align*}
\par 
Hence, we can   conclude that $\sigma_s(z)=y_m$. However, $z \vert_{G \setminus E_m} = x_m \vert_{G \setminus E_m}$ by the relation  \eqref{e:post-surjectivity-1-2-3} as $E_k \subset E_m$. Therefore, we obtain a contradiction to the choice of $x_m$, $y_m$. The proof is thus complete. 
\end{proof}

\par 
As a consequence of Lemma~\ref{l:uniform-psot-surjectivity} in the case of CA, we obtain the following result proved in   \cite[Lemma~1]{kari-post-surjective}. 

\begin{corollary}
Let $G$ be a countable group and let $A$ be a finite set. Suppose that $\tau \colon A^G \to A^G$ is a post-surjective CA. Then there exists a finite subset $E\subset G$ such that for all $x, y\in A^G$ with  $y\vert_{G \setminus \{1_G\}} =\sigma_s(x)\vert_{G \setminus \{1_G\}}$, there exists $z \in A^G$ such that $\sigma_s(z)=y$ and 
$z\vert_{G \setminus E}= x\vert_{G \setminus E}$. 
\end{corollary}

\begin{proof}
It is a direct consequence of Lemma~\ref{l:uniform-psot-surjectivity} and the fact that $\tau$ is a $G$-equivariant CA. 
\end{proof}
\par 
Using Lemma~\ref{l:uniform-psot-surjectivity} instead of \cite[Corollary~2]{kari-post-surjective}, we see easily that the exact same proof, \textit{mutatis mutandis}, of \cite{kari-post-surjective} shows that 
all pre-injective  post-surjective ANUCA with finite memory are  invertible. 

\begin{theorem}
\label{t:invertible-general}
Let $G$ be a countable group and let $A$ be a finite set. Suppose that $\tau \colon A^G \to A^G$ is a post-surjective  pre-injective ANUCA with finite memory. Then $\tau$ is invertible.  \qed
\end{theorem}

\section{Stably post-surjective ANUCA} 
\label{s:stably-post0surjective}

As for stable injectivity, we introduce the following notion of stably post-surjective ANUCA. 

\begin{definition}
Let $M$  be a subset of a group $G$ and let $S = A^{A^M}$ where $A$ is a set. 
Given $s \in S^G$, the ANUCA $\sigma_s$ is said to be  \emph{stably post-surjective} if for every $p \in \Sigma(s)$, the ANUCA $\sigma_p$ is post-surjective.
\end{definition}
\par 
With the above notation,  suppose that $\sigma_s$ is stably post-surjective for some $s \in S^G$. It is then  immediate that  $\sigma_p$ is also stably post-surjective for every $p \in \Sigma(s)$. It suffices to observe that 
$\Sigma(p) \subset \Sigma(s)$ (see the proof of Lemma~\ref{l:pseudo-equivarian}). 
\par 
Observe also that for every constant configuration $c \in S^G$, the subset $\Sigma(s) \subset S^G$ reduces to the single configuration $c$. Consequently, 
it follows immediately from the above definition that every post-surjective CA is stably post-surjective. 
\par 
We now state and prove the fundamental uniform post-surjectivity property of stably post-surjective ANUCA. 

\begin{lemma}
[Uniform post-surjectivity] 
\label{l:uniform-post-surjectivity-all} 
Let $M$  be a finite subset of a  countable group $G$. Let $A$ be a finite set and let $S=A^{A^M}$. Suppose that $\sigma_s$ is stably  post-surjective for some $s \in S^G$. Then there exists a finite subset $E\subset G$ such that for all $g \in G$ and $x, y\in A^G$ with $y\vert_{G \setminus \{g\}} =\sigma_s(x)\vert_{G \setminus \{g\}}$, there exists $z \in A^G$ such that $\sigma_s(z)=y$ and 
$z\vert_{G \setminus gE}= x\vert_{G \setminus gE}$.
\end{lemma}

\begin{proof}
By Lemma~\ref{l:uniform-psot-surjectivity}, we can find a finite subset $E$ such that  for all $x, y\in A^G$ with $y\vert_{G \setminus \{1_G\}} =\sigma_s(x)\vert_{G \setminus \{1_G\}}$, there exists $z \in A^G$ such that $\sigma_s(z)=y$ and 
$z\vert_{G \setminus E}= x\vert_{G \setminus E}$.
\par 
We can clearly suppose that $1_G\in M$ and $M$ is symmetric, i.e., $M=M^{-1}$. Moreover, 
since $G$ is countable, we can find an increasing sequence $(E_n)_{n \in \N}$ of finite subsets of $G$ such that $G= \cup_{n \in \N}E_n$,  $1_G \in E_0$, and for every $n \in \N$: 
\[
E_n M \subset E_{n+1}. 
\] 
\par 
We suppose on the contrary that for every $n \in \N$, there exists $g_n \in G$ and $u_n,v_n \in A^G$ such that $v_n\vert_{G \setminus \{g_n\}} =\sigma_s(u_n)\vert_{G \setminus \{g_n\}}$ but for all $z \in A^G$ with   $\sigma_s(z)=v_n$, one must have  
$z\vert_{G \setminus g_nE_n} \neq  u_n\vert_{G \setminus g_nE_n}$.  
\par 
For $n \in \N$, let us denote 
$s_n= g_n^{-1}s$, 
$x_n= g_n^{-1} u_n$, and 
$y_n= g_n^{-1} v_n$. Then  Lemma~\ref{l:pseudo-equivarian} implies that 
$y_n\vert_{G \setminus \{1_G\}} =\sigma_{s_n}(x_n)\vert_{G \setminus \{1_G\}}$ and for  $z \in A^G$ with   $\sigma_{s_n}(z)=y_n$, one has 
$z\vert_{G \setminus E_n} \neq  x_n\vert_{G \setminus E_n}$.  
\par 
Since $A^G$ and $S^G$ are  compact by Tychonoff's theorem, we can pass  to a subsequence and suppose without loss of generality that 
$\lim_{n \to \infty} x_n=x$ and  $\lim_{n \to \infty}y_n=y$ for some $x,y \in A^G$ and $\lim_{n \to \infty}s_n=p$ for some $p \in \Sigma(s)$. 
\par 
Consequently, by taking the limit of the relation $y_n\vert_{G \setminus \{1_G\}} =\sigma_{s_n}(x_n)\vert_{G \setminus \{1_G\}}$  when $n$ goes to $\infty$, we deduce from Lemma~\ref{l:continuity-x}   that 
\begin{align}
\label{e:stable-post-surjective-1-2-3} 
y\vert_{G \setminus \{1_G\}} =\sigma_{p}(x)\vert_{G \setminus \{1_G\}}. 
\end{align}
\par 
It follows from the stable post-surjectivity of $\sigma_s$ that $\sigma_p$ is also post-surjective. Hence, we deduce from \eqref{e:stable-post-surjective-1-2-3} that there exists a configuration $w \in A^G$ asymptotic to $x$ such that $\sigma_p(w)=y$. In particular, we have $w\vert_{G \setminus E_k} = x \vert_{G \setminus E_k}$ for some $k \in \N$. 
\par 
Now choose $m >k$ large enough so that 
\[
x_m \vert_{E_kM^2} = x  \vert_{E_k M^2}, \quad y_m \vert_{E_kM} = y  \vert_{E_k M}, \quad s_m \vert_{E_kM} = p \vert_{E_kM}. 
\]
\par 
It follows that $x_m \vert_{E_kM^2 \setminus E_k} = w  \vert_{E_k M^2 \setminus E_k}$ and we can define $z \in A^G$ by: 
\[
z \vert_{E_kM^2} = w \vert_{E_kM^2}, \quad z \vert_{G \setminus E_k}= x_m\vert_{G \setminus E_k}. 
\] 
\par 
Consequently, we deduce from  the formula  \eqref{e:induced-local-maps}-\eqref{e:induced-local-maps-general} that: 
\begin{align*}
    \sigma_{s_m}(z)\vert_{G \setminus E_kM} & = \sigma_{s_m}(x_m)  \vert_{G \setminus E_kM} & (\text{as }z \vert_{G \setminus E_k}= x_m\vert_{G \setminus E_k})\\
    & = y_m \vert_{G \setminus E_kM} & (\text{as } 1_G \in E_kM). 
\end{align*}
\par 
On the other hand, we find  that: 
\begin{align*}
    \sigma_{s_m}(z) \vert_{E_k M} & = \sigma_{s_m}(w)\vert_{E_k M} & (\text{as }z \vert_{E_kM^2} = w \vert_{E_kM^2} )\\ 
    & = \sigma_p(w)\vert_{E_k M} & (\text{as } s_m \vert_{E_kM} = p \vert_{E_kM}) \\
    & = y \vert_{E_k M} & (\text{as } \sigma_p(w)=y) \\
    & = y_m \vert_{E_k M}. &  
\end{align*}
\par 
Therefore, we deduce that $ \sigma_{s_m}(z) = y_m$. However, since  by construction  $z \vert_{G \setminus E_m} = x_m\vert_{G \setminus E_m}$ as $E_k \subset E_m$, we obtain a contradiction to the choice of $x_m$ and $y_m$. The proof of the lemma is thus complete. 
\end{proof}

\par 

The next results imply that the above uniform post-surjectivity is a stable property when passing to the limit of the configurations of local defining maps.

\begin{theorem}
\label{t:stably-uniform-post-surj-aysn-main}
Let $M$  be a finite subset of a countable group $G$. Let $A$ be a finite set and let $s \in S^G$ where  $S=A^{A^M}$. Suppose that $\sigma_s$ is stably  post-surjective. Then there exists $E \subset G$ finite such that for all $p \in \Sigma(s)$, $g \in G$, and $x, y\in A^G$ with  $y\vert_{G \setminus \{g\}} =\sigma_p(x)\vert_{G \setminus \{g\}}$, there exists $z \in A^G$ such that $\sigma_p(z)=y$ and 
$z\vert_{G \setminus gE}= x\vert_{G \setminus gE}$. 
\end{theorem}
\par 
\begin{proof}
Let $E \subset G$ be the subset given by Lemma~\ref{l:uniform-post-surjectivity-all} and let $p \in \Sigma(s)$. Then we can find a sequence $(g_n)_{n \in \N}$ in $G$ such that $\lim_{n \to \infty} g_n s = p$. 
\par 
Let $g \in G$ and $x, y\in A^G$ such that $y\vert_{G \setminus \{g\}} =\sigma_p(x)\vert_{G \setminus \{g\}}$. For every $n \in \N$, we define $y_n \in A^G$ be setting 
\[
y_n\vert_{G \setminus \{g\}}  =  \sigma_{g_ns} (x) \vert_{G \setminus \{g\}} , \quad y_n(g)= y(g)
\]
\par 
Since $\lim_{n \to \infty} g_n s = p$, we deduce from Lemma~\ref{l:continuity} that:
\begin{align*}
   \lim_{n \to \infty} y_n\vert_{G \setminus \{g\}}  = \lim_{n \to \infty}
    \sigma_{g_ns} (x)\vert_{G \setminus \{g\}}  = \sigma_p (x)\vert_{G \setminus \{g\}} = y \vert_{G \setminus \{g\}}. 
\end{align*}
\par 
In particular, we have  $\lim_{n \to \infty} y_n= y$ since $y_n(g)=y(g)$ for all $n \in \N$.  
\par 
Fix $n \in \N$. We infer from  Lemma~\ref{l:pseudo-equivarian} that 
$ (g_n \sigma_{s}(g_n^{-1}x))\vert_{G \setminus \{g\}} = y_n \vert_{G \setminus \{g\}}$. It follows that $\sigma_s(g_n^{-1}x)\vert_{G \setminus \{g_n^{-1}g\}} = (g_n^{-1}y_n)\vert_{G \setminus \{g_n^{-1}g\}}$. Hence, by Lemma~\ref{l:uniform-post-surjectivity-all}, we can find $t \in A^G$ such that $t\vert_{G \setminus g_n^{-1}gE}= (g_n^{-1}x )\vert_{G \setminus g_n^{-1}gE}$ and $\sigma_s(t) = g_n^{-1}y_n$. 
\par 
Let $z_n = g_n t$ then we deduce again  from Lemma~\ref{l:pseudo-equivarian} that  
\begin{equation}
\label{e:z-stable-post-surjective-3-2}
z_n \vert_{G \setminus gE}= x \vert_{G \setminus gE}, \quad \sigma_{g_n s}(z_n)= y_n.  
\end{equation}
\par 
Since $A^G$ is compact, we can pass to a subsequence and suppose without loss of generality that $\lim_{n \to \infty} z_n= z$ for some $z \in A^G$. It follows from \eqref{e:z-stable-post-surjective-3-2} that $z \vert_{G \setminus gE}= x \vert_{G \setminus gE}$. 
\par 
Moreover, since $\lim_{n \to \infty} y_n= y$, we have $\sigma_{p}(z)=y$ by Lemma~\ref{l:continuity-x}. This proves that $E$ satisfies the required condition in the conclusion of the theorem and the proof is thus complete.  
\end{proof} 

\section{Counter-examples} 
\label{s:counter}

 We present a simple  example of an ANUCA which is injective but not stably injective, not surjective, and not reversible. Hence, we obtain  counter-examples to Theorem~\ref{t:intro-characterization-reversible-ANUCA}  and Theorem~\ref{t:reversible-finite-alpha} when we replace the stable injectivity hypothesis by the weaker injectivity hypothesis.  
 
 \begin{example}
 \label{ex:1}
Let $G= \Z$ and let $A=\{0,1\}$. Let $M=\{-1,0,1\}$ and consider the functions $f,g \colon A^M \to A$ defined for all $(u,v,w)\in A^M$ by the following formula: 
\begin{equation}
    \label{e:counter-1-2}
f(u,v,w) \coloneqq w,  \quad g(u,v,w) \coloneqq  u+v \quad (\text{mod }2). 
\end{equation}
\par 
Let $S=A^{A^M}$ and let $p,q \in S^\Z $ where $p(n)=f$ and $q(n)=g$ for all $n\in \Z$. For every $k \in \Z$, we  define $s_k \in S^\Z$  by setting $s_k(n)=f$ if $n \leq k$ and $s(n)=g$ if $n \geq k+1$.  
\par 
Denote $s=s_0$ then it is clear that $\Sigma(s)=\{ s_k\colon k \in \Z \}\cup \{p,q\}$. We claim that the ANUCA $\sigma_s\colon A^\Z \to A^\Z$ is injective but not stably injective. 
\par
Indeed, suppose that $\sigma_s(x)=\sigma_s(y)$ for some $x,y \in A^\Z$. Then we deduce from 
\eqref{e:counter-1-2} that $x(n)=y(n)$ for all $n \leq 1$ and $x(n)+x(n+1)=y(n)+y(n+1)$ (mod 2) for all $n \geq 0$. It follows immediately that $x(n)=y(n)$ for all $n \in \Z$. We conclude that $x=y$ and thus $\sigma_s$ is injective. On the other hand, observe that $\sigma_q$ is not injective since $\sigma_q(0^\Z)=\sigma_s(1^\Z)=0^\Z$. As $q \in \Sigma(s)$, we conclude that $\sigma_s$ is not stably injective. 
\par 
We claim that $\sigma_s$ is not surjective. Indeed, let $c \in A^\Z$ be the configuration given by $c(n)=0$ for all $n \leq 0$ and $c(n)=1$ for all $n \geq 1$. Suppose on the contrary that $\sigma_s(x)=c$ for some $x \in A^\Z$. It follows from the definition of $s$ and \eqref{e:counter-1-2} that 
$x(n)=c(n-1)=0$ for all $n \geq 1$ and $x(0)+x(1)= c(1)$. Hence  $c(1)=0$ and we obtain a contradiction. 
\par 
Finally, let $x \in A^G$ and put $y=\sigma_s(x)$. By a direct induction, we infer from the relations $y(n)=x(n+1)$ for $n \leq 0$ and $y(n)=x(n)+x(n-1)$ (mod 2) for $n \geq 1$ that for all $n \geq 2$, we have: 
\[
x(n)= 
y(n) - y(n-1) + \dots + (-1)^{n+1} y(-1) \text{ (mod }2) 
\]
and that $y\vert_{\N_{\geq 2}}$ can take any value in $A^{\N_{\geq 2}}$ (in fact, the only requirement for $y$ is that $y(1)=y(0)+y(-1)$).  Hence, $x(n)$ must depend on $y(n), \dots, y(2)$ for all $n \geq 2$. Consequently, if   $\tau \colon A^\Z \to A^\Z$ is an ANUCA such that $\tau \circ \sigma_s=\Id$ then $\tau$ cannot have finite memory and we conclude that $\sigma_s$ is not reversible (see Section~\ref{s:reversible-stablyinjective}).  
\end{example} 

\par 
The next similar example shows that unlike CA, a bijective ANUCA is not necessarily reversible or stably injective.
\begin{example} 
\label{ex:2}
Let $G= \Z$ and let $A=\{0,1\}$. Let $M=\{-1,0\}$ and consider the functions $f,g \colon A^M \to A$ defined for all $(u,v,w)\in A^M$ by the following formula: 
\begin{equation}
    \label{e:counter-1-2-3}
f(u,v) \coloneqq v,  \quad g(u,v) \coloneqq  u+v \quad (\text{mod }2). 
\end{equation} 
\par 
Let $S=A^{A^M}$ and for $k \in \Z\cup \{\pm \infty\}$, we  define $s_k \in S^\Z$  by setting $s_k(n)=f$ if $n \leq k$ and $s(n)=g$ if $n \geq k+1$.  
Denote $s=s_0$ then we claim  that $\sigma_s\colon A^\Z \to A^\Z$ is bijective but it is not reversible and thus not stably injective by Theorem~\ref{t:reversible-finite-alpha}. 
\par 
Indeed,  let $x \in A^G$ and  $y=\sigma_s(x)$. As in Example~\ref{ex:1}, we infer from \eqref{e:counter-1-2-3} that  $x(n)=y(n)$ for $n \leq 0$ and  we have for all $n \geq 1$: 
\[
x(n)= 
y(n) - y(n-1) + \dots + (-1)^{n} y(0) \text{ (mod }2).  
\]
\par 
It follows immediately that $\sigma_s$ is bijective but it is not reversible since $x(n)$ depends on $y(0), \dots, y(n)$ for all $n \geq 1$. The fact that $\sigma_s$ is not stably injective can also be seen directly by checking that $\sigma_{s_{-\infty}}$ is not injective. 
\end{example}
\par 
Finally, we present an example of a very simple one-dimensional stably injective ANUCA which is not surjective. 

\begin{example}
\label{ex:3} 
Let $G= \Z$ and let $A=\{0,1\}$. Let $M=\{-1,0,1\}$ and consider the functions $f,g, h \colon A^M \to A$ defined for all $(u,v,w)\in A^M$ by the following formula: 
\begin{equation}
    \label{e:counter-1-2-3}
f(u,v,w) \coloneqq w,  \quad g(u,v,w) \coloneqq  u, \quad h(u,v,w)=v \quad (\text{mod }2). 
\end{equation} 
\par 
Let $S =A^{A^M}$ and consider $s \in S^\Z$ defined by $s(n)=f$ if $n \leq -1$, $s(n)=h$ if $n=0$, and $s(n)=g$ if $n \geq 1$. 
Then $\sigma_s$ is injective since 
$\sigma_s(x)=y$ implies that 
$x(n)=y(n-1)$ for $n \leq 0$ and $ x(n)=y(n+1)$ for $n \geq 0$. 
\par Let $p, q\in S^\Z$ where $p(n)=f$ and $q(n)=g$ for all $n \in \N$ then it is clear that $\sigma_p$ and $\sigma_q$ are injective. Since $\Sigma(s)= \{p,q,s\}$, we deduce that $\sigma_s$ is stably injective. 
On the other hand, $\sigma_s$ is not surjective since we can check directly that  
\[\im \sigma_s = \{y \in A^\Z\colon y(-1)=y(0)=y(1)\}.
\]
\end{example}

\bibliographystyle{siam}

\begin{thebibliography}{10}
  




\bibitem{bartholdi-kielak}
{\sc L.~Bartholdi}, 
{\em Amenability of groups is characterized by Myhill's Theorem. With an appendix by D. Kielak}, 
J. Eur. Math. Soc. vol. 21, Issue 10 (2019), pp. 3191--3197. 

\bibitem{kari-post-surjective}
{\sc S.~Capobianco, J.~Kari, S.~Taati},
{\em An “almost dual” to Gottschalk’s Conjecture}. 22th International Workshop on Cellular Automata and Discrete Complex Systems (AUTOMATA), Jun 2016, Zurich, Switzerland. pp.77--89
 
  
\bibitem{csc-sofic-linear}
{\sc T.~Ceccherini-Silberstein and M.~Coornaert}, 
{\em Injective linear cellular automata and sofic groups}, 
Israel J. Math. 161 (2007), pp.~1--15. 

 
  
\bibitem{hedlund-csc}
\leavevmode\vrule height 2pt depth -1.6pt width 23pt, 
A generalization of the Curtis-Hedlund theorem, Theoret. Comput. Sci., 400 (2008), pp. 225–229
 
 
 

\bibitem{cscp-alg-goe}
{\sc T.~Ceccherini-Silberstein, M.~Coornaert, and X.~K. Phung},  
{\em On the Garden of Eden theorem for endomorphisms of symbolic algebraic varieties}, 
 Pacific J. Math. 306 (2020), no. 1, pp~31--66. 


\bibitem{ceccherini}
{\sc T.~Ceccherini-Silberstein, A.~Mach{\`{\i}}, and F.~Scarabotti}, {\em
  Amenable groups and cellular automata}, Ann. Inst. Fourier (Grenoble), 49
  (1999), pp.~673--685.


\bibitem{Den-12a}  
{\sc A. Dennunzio, E. Formenti, and J. Provillard},  {\em Non-uniform cellular automata: Classes, dynamics, and decidability},  
Information and Computation
Volume 215, June 2012, Pages 32-46
 
\bibitem{Den-12b} 
{\sc A. Dennunzio, E. Formenti, and J. Provillard},   {\em Local rule distributions, language complexity and non-uniform
cellular automata}. Theoretical Computer Science, 504  (2013) 38–51. 



\bibitem{elek}
{\sc G.~Elek and A.~Szab\' o}, 
{\em Sofic groups and direct finiteness}, 
J. of Algebra 280 (2004), pp.~426-434. 


 \bibitem{folner}
 {\sc E. Følner}, 
 {\em On groups with full Banach mean value}, 
 Math. Scand. 3 (1955), pp.~243--254.
 
 
 \bibitem{GOL} 
{\sc M.~Gardner}. 
{\em Mathematical Games: The Fantastic Combinations of John Conway’s New Solitaire Game “Life”}. Scientific American, 223 (4), 1970,  pp.~120--123.


\bibitem{gromov-esav}
{\sc M.~Gromov}, {\em Endomorphisms of symbolic algebraic varieties}, J. Eur.
  Math. Soc. (JEMS), 1 (1999), pp.~109--197.
 



\bibitem{gottschalk}
{\sc W.H.~Gottschalk}, 
{\em Some general dynamical notions}, 
Recent advances in topological dynamics, Springer, Berlin, 1973, pp. 120--125. Lecture Notes in Math. Vol. 318.

\bibitem{hedlund}  {\sc G. A. Hedlund}, 
{\em Endomorphisms and automorphisms of the shift dynamical system}, 
Math. Systems Theory, 3 (1969), pp.~320--375.


 


\bibitem{moore}
{\sc E.~F. Moore}, {\em Machine models of self-reproduction}, vol.~14 of Proc.
  Symp. Appl. Math., American Mathematical Society, Providence, 1963,
  pp.~17--34.

\bibitem{myhill}
{\sc J.~Myhill}, {\em The converse of {M}oore's {G}arden-of-{E}den theorem},
  Proc. Amer. Math. Soc., 14 (1963), pp.~685--686.
 



\bibitem{neumann-amenable}
{\sc J.~von Neumann}, 
{\em Zur allgemeinen Theorie des Masses}, Fund. Math. 13 (1929), pp.~73--116
and 333. = \emph{Collected works}, vol. I, pp~599--643.

\bibitem{neumann}
{\sc J.~von Neumann}, 
{\em The general and logical theory of automata}, Cerebral Mechanisms in Behavior. The Hixon Symposium, John Wiley \& Sons Inc., New York, N. Y., 1951, pp.~1--31; discussion, pp. 32--41.


\bibitem{phung-2020}
{\sc X.K.~Phung}, 
{\em On sofic groups, Kaplansky's conjectures, and endomorphisms of pro-algebraic groups}, Journal of Algebra, 562 (2020), pp.~537--586. 


\bibitem{phung-post-surjective}
\leavevmode\vrule height 2pt depth -1.6pt width 23pt,
{\em On symbolic group varieties and dual surjunctivity}, preprint.  arXiv:2111.02588

\bibitem{phung-geometric}
\leavevmode\vrule height 2pt depth -1.6pt width 23pt,
{\em A geometric generalization of Kaplansky's direct finiteness conjecture}, preprint. arXiv:2111.07930 

\bibitem{phung-weakly}
\leavevmode\vrule height 2pt depth -1.6pt width 23pt, 
{\em Weakly surjunctive groups and symbolic group varieties}, preprint. arXiv:2111.13607


\bibitem{weiss-sgds}
{\sc B.~Weiss}, 
{\em Sofic groups and dynamical systems}, 
Sankhy\=a Ser. A 62 (2000), no. 3, pp.~350--359. 
Ergodic theory and harmonic analysis (Mumbai, 1999).  

\bibitem{stan-amenable}
{\sc S. Wagon}, 
{\em The Banach-Tarski paradox}, Cambridge University Press, Cambridge, 1993. With a foreword by Jan Mycielski; Corrected reprint of the 1985 original.
 
\end{thebibliography}

\end{document}